\documentclass[onefignum,onetabnum]{siamart190516}

\usepackage{graphicx}	
\usepackage{amssymb}
\usepackage{empheq}
\usepackage{amsmath, amsfonts}
\usepackage{float}
\usepackage{subcaption}
\usepackage{multirow}
\usepackage{makecell}
\usepackage{caption}
\usepackage{enumitem}
\usepackage{color}
\usepackage{soul}
\usepackage{empheq}
\usepackage{multicol}
\usepackage{algorithm}
\usepackage[noend]{algpseudocode}
\usepackage{setspace}
\usepackage{mathtools}
\usepackage{comment}
\usepackage{bm}
\usepackage{amsfonts}
\usepackage{epstopdf}
\usepackage{standalone}

\usepackage{relsize}
\usepackage{tikz}
\usetikzlibrary{shapes}
\usetikzlibrary{plotmarks}
\usepackage{csvsimple}
\usepackage{filecontents}

\newsiamremark{remark}{Remark}
\newsiamremark{hypothesis}{Hypothesis}
\crefname{hypothesis}{Hypothesis}{Hypotheses}
\newsiamthm{claim}{Claim}

\usepackage{mathrsfs}

\newcommand{\beq}{\begin{equation}}
\newcommand{\eeq}{\end{equation}}

\newcommand{\ep}{{\epsilon}}

\newcommand{\del}{\delta}

\newcommand{\Om}{\Omega}

\def\cI{\mathcal{I}}

\def\cL{\mathcal{L}}
\def\cI{\mathcal{I}}
\def\cS{\mathcal{S}}

\def\R{\mathbb{R}}
\def\Z{\mathbb{Z}}

\def\nd{{\textnormal{d}}}

\begin{document}

\headers{RK Collocation for NONLOCAL DIFFUSION}{Yu Leng, Xiaochuan Tian, Nathaniel Trask, and John.T. Foster}

\title{Asymptotically compatible reproducing kernel collocation and meshfree integration for nonlocal diffusion 
}

\author{Yu Leng 
\thanks{Department of Petroleum and Geosystems Engineering, The University of Texas at Austin, Austin, TX 78712 (\email{yu-leng@utexas.edu}, \email{john.foster@utexas.edu}). The work of these authors is supported in part by the AFOSR MURI Center for Material Failure Prediction through Peridynamics (AFOSR Grant NO. FA9550-14-1-0073) and the SNL:LDRD academic alliance program.}
\and Xiaochuan Tian
\thanks{Department of Mathematics, University of California, San Diego, CA 92093
(\email{xctian@ucsd.edu}). The work of this author is supported in part by NSF  grant DMS-1819233 and DMS-2044945.}
\and Nathaniel Trask
\thanks{Center for Computing Research, Sandia National Laboratories, Albuquerque, NM
(\email{natrask@sandia.gov}). Sandia National Laboratories is a multimission laboratory managed and operated by National Technology and Engineering Solutions of Sandia, LLC., a wholly owned subsidiary of Honeywell International, Inc., for the U.S. Department of Energy’s National Nuclear Security Administration under contract DE-NA-0003525. This paper describes objective technical results and analysis. Any subjective views or opinions that might be expressed in the paper do not necessarily represent the views of the U.S. Department of Energy or the United States Government. }
\and John.T. Foster \footnotemark[1] }

\maketitle

\begin{abstract}
Reproducing kernel (RK) approximations are meshfree methods that construct shape functions from sets of scattered data. We present an asymptotically compatible (AC) RK collocation method for nonlocal diffusion models with Dirichlet boundary condition. 
The numerical scheme is shown to be convergent to both nonlocal diffusion and its corresponding local limit as nonlocal interaction vanishes. 
The analysis is carried out on a special family of rectilinear Cartesian grids for linear RK method with designed kernel support. 
The key idea for the stability of the RK collocation scheme is to compare the collocation scheme with the standard Galerkin scheme which is stable. 
In addition, assembling the stiffness matrix of the nonlocal problem requires costly computational resources because 
high order Gaussian quadrature is necessary to evaluate the integral. We thus provide a remedy to the problem by introducing a quasi-discrete nonlocal diffusion operator for which no numerical quadrature is further needed after applying the RK collocation scheme. The quasi-discrete nonlocal diffusion operator combined with RK collocation is shown to be convergent to the correct local diffusion problem by taking the limits of nonlocal interaction and spatial resolution simultaneously. The theoretical results are then validated with numerical experiments. We additionally illustrate a connection between the proposed technique and an existing optimization based approach based on generalized moving least squares (GMLS).

\end{abstract}




\section{Introduction}

This work is motivated by the study of numerical solutions to linear nonlocal models and their local limits. Peridynamics (PD) is a nonlocal theory of continuum mechanics  \cite{Silling2000}. Unlike the classical theory, PD models are formulated using spatial integration instead of differentiation, making them well-suited for describing discontinuities such as fracture, material separation and failure. PD has been applied to hydraulic-fracture propagation problems \cite{Ouchi2017}, crack branching \cite{Bobaru2012}, damage progression in multi-layered glass \cite{Ha2010} and others. Linear PD models also share similarities with nonlocal diffusion model \cite{Du2012}. 
Rigorous mathematical analysis and a variety of numerical methods have been developed for PD and nonlocal diffusion models \cite{bobaru2016handbook,XChen2011,du2019nonlocal,Du2012,Du2013nonlocal,du2013posteriori,mengesha2014nonlocal,Seleson2016C,Silling2005,Tian2014,Trask2019}.
Nonlocal models introduce a length scale $\delta$, called the horizon, which takes into account nonlocal interactions. As $\delta$ goes to zero nonlocal interactions vanish and nonlocal models recover their local equivalents, provided that the limit is well-defined.  It is a common practice to couple $\delta$ with the mesh size $h$ in engineering applications, but some standard numerical methods may converge to wrong local limits \cite{Tian2013}. The investigation of local limits of numerical schemes is of fundamental importance, because it encodes the robustness of the numerical methods for nonlocal models with a changing parameter $\delta$. 

A mathematical framework of convergence is established for PD and nonlocal diffusion models in \cite{Du2012,Mengesha2014,mengesha2014nonlocal} and asymptotically compatible (AC) discretization is introduced in \cite{Tian2013,Tian2014}. The AC scheme allows the numerical solution of nonlocal equations to converge to both the nonlocal solutions for a fixed $\delta$ and also their local limits as $\delta$ goes to zero, independent of the mesh size $h$.
The study of AC schemes  has  since then been developed for various numerical methods and model problems \cite{chen2017asymptotically,chen2015selecting,Du2016,Yang2018,lee2019asymptotically,tian2017conservative,Trask2019,zhang2018accurate}. Finite element methods (FEM) for nonlocal equations are studied in \cite{Tian2013,Tian2014} and FEM with subspaces containing piecewise linear functions are shown to be AC. However, applying FEM to nonlocal problems is computationally prohibitive because the variational formulation of nonlocal equations involves  a double integral and costly  geometrically mesh intersection calculation  \cite{XChen2011,Xu2016b}. Further, the nonlocal kernels in PD models are often singular, which adds more complexity to the computation. Finite difference methods (FDM) do not need the evaluation of a double integral but require uniform grids to obtain both AC and discrete maximum principle at the same time \cite{Yang2018}. 
A meshfree discretization \cite{Silling2005} of PD equations is widely used in engineering applications due to its simplicity. This meshfree method uses a set of particles in the domain, each with a known volume, and it assumes constant fields in each nodal element. This method, however, suffers from large integration error leading to low order of convergence and it is not robust under the change of the horizon parameter. Later, more works \cite{Seleson2016,Yu2011,Trask2019} have been devoted to improve the integration error, but rigorous numerical analysis falls behind.  A reproducing kernel (RK) collocation approach is proposed and numerically studied in \cite{Pasetto2018}. However, the convergence and robustness of the method needs further investigation. 

The first motivation of this work is to provide a convergence analysis of RK collocation method for nonlocal diffusion models. Stability of collocation methods on integral equations is not a trivial task, due to the lack of a discrete maximum principle.
A helpful view is to compare collocation schemes with Galerkin schemes \cite{arnold1984asymptotic,arnold1983asymptotic,Costabel1992, hu2011perturbation, hu2011error}, for which stability comes naturally. In this work, we use the Fourier approach \cite{Costabel1992} and demonstrate that the Fourier symbol of the RK collocation scheme with linear interpolation order and suitable choice of $\delta$ for nonlocal diffusion can be bounded below by that of the standard Galerkin scheme on Cartesian grids. Consequently, we show that the collocation scheme is stable because the standard Galerkin approximation is uniformly stable (i.e. the stability constant does not depends on $\delta$). The consistency of the scheme is established using the approximation properties of the RK approximation, and it is shown that the truncation error is independent of the model parameter $\delta$. 
Therefore the proposed RK collocation method on nonlocal diffusion is AC.

Although the collocation scheme requires only a single integration to be performed for the evaluation of each nonzero entry in the stiffness matrix, it is still quite expensive, particularly for models with singular kernels. In practice, high-order Gauss quadrature rules are used to evaluate the integral \cite{XChen2011,Pasetto2018}. 
Therefore, the second goal of this work is to develop a practical numerical method for nonlocal models. To this end, we introduce a quasi-discrete nonlocal diffusion operator which replaces the integral with a finite summation of quadrature points inside the horizon. We utilize the RK technique to calculate the quadrature weights.  The quasi-discrete nonlocal diffusion operator discretized with the collocation scheme saves computational cost and it could potentially be useful for fracture problems involving bond breaking \cite{Ha2010}. 
A similar technique has been proposed in \cite{Trask2019} utilizing an optimization construction which admits interpretation as a generalized moving least squares (MLS) process. It is well known that RK and MLS shape functions are equivalent, up to a rescaling of the weighting function and for particular reproducing spaces \cite{chen2017meshfree}. We will show that the construction of quadrature weights using the RK technique is similarly equivalent under certain conditions to this generalized MLS approach \cite{Trask2019}, and therefore the stability proof provided here applies equally to this second class of schemes which currently lack a proof of stability. This unifies existing work in the literature using both RK \cite{Pasetto2018} and MLS \cite{Trask2019} as a framework to develop AC particle-based schemes.

This paper is organized as follows. In \cref{sec:NonlocalEquation}, we introduce the nonlocal diffusion model equations with Dirichlet boundary conditions. In \cref{sec:RKCollocation}, we present the RK collocation method  and work on linear interpolation order with special choices of RK support sizes.  \Cref{sec:CAnalysis} discusses the convergence of the 
RK collocation method 
to both the nonlocal diffusion equation for a fixed $\delta$ and the local diffusion equation as $\delta$ goes to zero. As a result, the RK collocation scheme is AC. Then, a quasi-discrete nonlocal diffusion operator is developed in \cref{sec:QuasiDiscrete} and its convergence analysis is presented in \cref{sec:DCAnalysis}.  \Cref{sec:NumericalExample} gives numerical examples to complement our theoretical analysis. Finally, conclusions are made in \cref{sec:Conclusion}.

\section{Nonlocal diffusion operator and model equation} \label{sec:NonlocalEquation}
We use the following notation throughout the paper. The spatial dimension is denoted as d, which is a positive integer. A generic point $\bm{x} \in \mathbb{R}^{\nd}$ is expressed as $\bm{x}=(x_1, \ldots, x_{\nd})$. A multi-index is a collection of d nonnegative integers, $\bm{\alpha} = (\alpha_1, \ldots, \alpha_{\nd})$ and its length is $|\bm{\alpha}| = \sum_{i=1}^{\nd}\alpha_i$. For a given $\bm{\alpha}$, we write $\bm{x}^{\bm{\alpha}}=x_1^{\alpha_1}\ldots x_{\nd}^{\alpha_{\nd}}$. Let $\Omega \subset \mathbb{R}^{\nd}$ be a bounded, open domain. The corresponding interaction domain is then defined as  
\[
\Omega_{\cI} = \{ \bm{x} \in \mathbb{R}^{\nd} \backslash \Omega :  \text{dist}(\bm{x}, \Omega) \leq \delta\}\, ,
\] 
and let $\Omega_{\delta} = \Omega \cup \Omega_{\cI}$. Following the same notations as in \cite{Du2012}, we define the nonlocal diffusion operator  $\mathcal{L}_{\delta}$,
for a given $u(\bm{x}): \Omega_{\delta} \rightarrow \mathbb{R}$, as 
\begin{equation}  \label{eqn:NonlocalOper}
\mathcal{L}_{\delta} u(\bm{x}) = \int_{\Omega_\delta}  \rho_{\delta}(\bm{x},\bm{y})(u(\bm{y})-u(\bm{x}))d\bm{y}, \quad \forall \, \bm{x} \in \Omega,
\end{equation}
where $\delta$ is the nonlocal length and $\rho_{\delta} (\bm{x},\bm{y})$ is the nonlocal diffusion kernel which is nonnegative and symmetric, i.e., $\rho_{\delta}(\bm{x}, \bm{y}) = \rho_{\delta}(\bm{y}, \bm{x})$. Let us consider a nonlocal diffusion problem with homogeneous Dirichlet volumetric constraint, 

\begin{equation}  \label{eqn:NonlocalEqn}
\begin{cases}
-\mathcal{L}_{\delta}u = {f_\delta}, & \textnormal{in } \Omega, \\
 \quad \quad  u = 0, &  \textnormal{on } \Omega_{\cI}.
\end{cases}
\end{equation}
{Notice that we allow the source data to be dependent on $\delta$.}

In this work, we study the kernels of radial type, i.e., $\rho_{\delta}(\bm{x},\bm{y}) = \rho_{\delta}(|\bm{x}-\bm{y}|)$. In addition, we assume
\begin{equation} \label{eqn:NonlocalKernelScaling}
\rho_{\delta}(|\bm{s}|)=\frac{1}{\delta^{{\nd}+2}}\rho \left(\frac{|\bm{s}|}{\delta} \right),
\end{equation}
where $\rho (|\bm{s}|)$ is compactly supported in $B_{1}(\bm{0})$ (the unit ball about $\bm{0}$).  We further assume $\rho (|\bm{s}|)$ is a non-increasing function  and  has a bounded second-order moment, i.e.
\begin{equation}  \label{eqn:boundedmoment}
\int_{B_{\delta}(\bm{0})}\rho_{\delta}(|\bm{s}|)|\bm{s}|^2 d\bm{s} = \int_{B_{1}(\bm{0})}\rho(|\bm{s}|)|\bm{s}|^2 d\bm{s} =2 {\nd}.
\end{equation}
The local limit of $\mathcal{L}_{\delta}$ is denoted as $\mathcal{L}_0$ when $\delta \rightarrow 0$. We are interested in particular cases where $\mathcal{L}_0 = \Delta$, such that \cref{eqn:NonlocalEqn}  goes to  
\begin{equation}  \label{eqn:LocalEqn}
\begin{cases}
-\mathcal{L}_{0}u = {f_0}, & \textnormal{in } \Omega, \\
 \quad \quad u = 0, & \textnormal{on } \partial \Omega.
\end{cases}
\end{equation}
In order for \cref{eqn:NonlocalEqn} to be convergent to \cref{eqn:LocalEqn} as $\del\to0$, we need the consistency of the source data. Here and in the rest of the paper, we assume that $f_\delta$ converges to $f_0$ uniformly in second order, {\it i.e.}, 
\beq \label{eqn:2ndorderassump}
\max_{\bm{x}\in\Om} |f_\delta(\bm{x}) - f_0(\bm{x})| = O(\delta^2) \,.
\eeq

We proceed to define some functional spaces. The natural energy space and the 
constrained energy space are defined as 
\begin{equation} \notag
\mathcal{S}_{\delta} :=\left \{ u \in L^2(\R^\nd): \int_{\R^\nd} \int_{\R^\nd}\rho_{\delta}(|\bm{y}-\bm{x}|)|u(\bm{y})-u(\bm{x})|^2d\bm{y} d\bm{x} < \infty \right \} 
\end{equation}
and 
\[
\cS_{c,\del}: =\{ u\in \cS_\del: u(\bm{x})=0,\, \forall \bm{x} \in \R^\nd \backslash\Om\} 
\]
respectively. 
The nonlocal diffusion problem  \cref{eqn:NonlocalEqn} is well-posed with weak solutions in the constrained energy space $\cS_{c,\del}$. The well-posedness of the problem is a result of Lax-Milgram theorem and the nonlocal Poincar\'e inequality established in \cite{Du2013nonlocal, Mengesha2014}. 
The following uniform stability is a result of the uniform nonlocal Poincar\'e inequality shown in \cite{Mengesha2014}.
\begin{lemma} \label{thm:nonlocalmapping}
\textbf{(Uniform stability)} Assume that $\widetilde{\Om}\subset \R^\nd$ is an open bounded and connected domain and $\delta \in (0, \delta_0]$ for some $\delta_0 > 0$. The bilinear form $(-\mathcal{L}_{\delta}u, u)$ is an inner product and  for any $u\in \mathcal{S}_{\delta}$ with $u|_{\R^\nd\backslash\widetilde{\Om}} = 0$, we have 
\begin{equation} \notag 
|(-\mathcal{L}_{\delta}u, u)| \geq C \| u \|^2_{L^2(\R^\nd)}\,,
\end{equation} 
where $C$ is a constant that only depends on $\widetilde{\Om}$ and $\delta_0$.
\end{lemma}

At last, we remark that we write \cref{eqn:NonlocalEqn} and  \cref{eqn:LocalEqn} as homogeneous Dirichlet boundary problems for the convenience of exposition. 
In fact, non-homogeneous Dirichlet problems can be easily cast into homogeneous ones given by \cref{eqn:NonlocalEqn} and  \cref{eqn:LocalEqn}. Indeed, assume that we have the following non-homogeneous Dirichlet problems
\beq \label{eqn:nonhomogeneuos}
\begin{cases}
-\mathcal{L}_{\delta}u ={f_\delta} & \textnormal{in } \Omega \\
 \quad \quad  u = g &  \textnormal{on } \Omega_{\cI}
\end{cases}
\qquad \text{and } \qquad 
\begin{cases}
-\mathcal{L}_{0}u = f_0 & \textnormal{in } \Omega \\
 \quad \quad u = g & \textnormal{on } \partial \Omega.
\end{cases}
\eeq
Assume that the boundary data $g$ can be smoothly extended to the domain $\Om\cup \Om_{\cI}$, {\it i.e.}, there exists $w \in C^4(\overline{\Om\cup \Om_{\cI}})$ such that $w|_{{\Om_{\cI}}}=g$. By letting $v=u-w$, we can rewrite \cref{eqn:nonhomogeneuos} into homogeneous Dirichlet problems
\beq \label{eqn:nonhomogeneuous_trans}
\begin{cases}
-\mathcal{L}_{\delta}v = {f_\delta} + \mathcal{L}_{\delta} w & \textnormal{in } \Omega \\
 \quad \quad  v = 0 &  \textnormal{on } \Omega_{\cI}
\end{cases}
\qquad \text{and } \qquad 
\begin{cases}
-\mathcal{L}_{0}v = f_0 +\mathcal{L}_{0}w & \textnormal{in } \Omega \\
 \quad \quad v = 0 & \textnormal{on } \partial \Omega.
\end{cases}
\eeq
Notice that the source data in \cref{eqn:nonhomogeneuous_trans} still satisfies the uniform second order convergence assumption, since for $w\in C^4(\overline{\Om\cup \Om_{\cI}})$, we have 
$|\mathcal{L}_{\delta} w(\bm{x})- \mathcal{L}_0 w(\bm{x})|=O(\del^2)$. 
Therefore in the rest of the paper, we assume the homogeneous boundary  conditions in \cref{eqn:NonlocalEqn} and  \cref{eqn:LocalEqn} together with the consistency assumption \cref{eqn:2ndorderassump} of the source data.

\section{RK collocation method} \label{sec:RKCollocation}

We first introduce some notations.  
Define  $\square$ to be a rectilinear Cartesian grid on $\mathbb{R}^{\nd}$, namely 
\begin{equation}  \notag
\square := \{ \bm{x_k} :=  \bm{k} \odot \bm{h} \mid \bm{k} \in \mathbb{Z}^{\nd} \},
\end{equation}
where $ \bm{k}=(k_1,k_2,\ldots,k_\nd)$ and $\bm{h}=(h_1, h_2, \ldots, h_{\nd}) $ consists of discretization parameters in each dimension and $\odot$ denotes component-wise multiplication, i.e.,
\begin{equation} \notag 
\bm{k} \odot \bm{h} = (k_1h_1, k_2h_2, \ldots, k_{\nd}h_{\nd}).
\end{equation}
Sometime we also write the $j$-th component of $\bm{x_k}$ by $x_{k_j} $, which is equal to $k_j h_j$ by definition. 
We  introduce a component-wise division symbol $\oslash$: 
\begin{equation} \notag 
\bm{k} \oslash \bm{h} = \left(\frac{k_1}{h_1}, \frac{k_2}{h_2}, \ldots, \frac{k_{\nd}}{h_{\nd}} \right).
\end{equation}
Note that the grid size $h_j$ can be different for different $j$. For instance, in two dimension, rectangular grids are allowed. Nonetheless, we assume the grid $\square$ is quasi-uniform such that $\bm{h}$ can also be written as
\begin{equation} \label{eqn:GridVector} 
\bm{h} =  h_{\max} \bm{\hat{h}}\,,
\end{equation}
with $\bm{\hat{h}}$ being a fixed vector with the maximum component to be 1. 
For any continuous function $u(\bm{x})$, define the restriction to $\square$ by
\begin{equation} 
r^hu : = (u(\bm{x_k}))_{\bm{k}\in \mathbb{Z}^{\nd}},
\end{equation}
and the restriction to ($\square \cap \Omega$) as 
\begin{equation} \label{restriction_omg}
r^h_{\Omega} u  := (u(\bm{x_k})), \quad \bm{x}_{\bm{k}} \in (\square \cap \Omega),
\end{equation}
where $\square \cap \Omega$ is the collection of grid points that only reside in $\Omega$.
For any sequence $(u_{\bm{k}})_{\bm{k}\in\mathbb{Z}^{\nd}}$ on $\mathbb{R}$, the RK interpolant operator is given as 
\begin{equation} \notag
i^h(u_{\bm{k}}) : = \sum_{\bm{k} \in \mathbb{Z}^{\nd}}  \Psi_{\bm{k}}(\bm{x}) u_{\bm{k}},
\end{equation}
where $\Psi_{\bm{k}}(\bm{x})$ is the RK basis function to be introduced shortly. Denote by $S(\square)$ the trial space equipped with the RK basis  $\Psi_{\bm{k}}(\bm{x})$ on $\square$, i.e., $S(\square) =\textnormal{span}\{\Psi_{\bm{k}}(\bm{x}) \mid \bm{k} \in \mathbb{Z}^{\nd}\} $. 
Let
\begin{equation} \notag 
\Pi^h := i^hr^h
\end{equation}
be the interpolation projector from the space of continuous functions on $\mathbb{R}^\nd$ to the trial space $S(\square)$.

We proceed to recall the construction of the RK basis function. The RK approximation \cite{MP:Liu1995} of $u(\bm{x}) : \mathbb{R}^{\nd} \to \mathbb{R}$ on $\square$ is formulated as:
\begin{equation} \label{eqn:DiscUh}
\Pi^h u(\bm{x}) =  \sum\limits_{\bm{k} \in \mathbb{Z}^{\nd}} C(\bm{x},\bm{x}-\bm{x_k}) \bm{\phi_a}(\bm{x}-\bm{x_k}) u(\bm{x_k}),
\end{equation}
where $C(\bm{x}; \bm{x}-\bm{y})$ is the correction function, $u(\bm{x_k})$ is the nodal coefficient, and $\bm{\phi_a}(\bm{x}-\bm{y})$ is the kernel function defined as the tensor product of kernel functions in each dimension with support $\bm{a}$, i.e.
\begin{equation} \label{eqn:CartesianWindow}
\bm{\phi_a}(\bm{x}-\bm{y}) \equiv \prod^{\nd}_{j=1}\phi_{a_j}(x_j-y_j)=\prod^{\nd}_{j=1} \phi \left( \frac{|x_j-y_j|}{a_j} \right),
\end{equation}
where $\phi_{a_j}(x_j) $ is the kernel function in the $j$-th dimension, $a_j$ is the support size for $\phi_{a_j}(x_j)$ and $\phi(x)$ is called the window function. In this work, we use the cubic B-spline function as the window function, i.e.,
\begin{equation}  \label{eqn:CubicSpline}
\phi(x) = 
\begin{cases}
\frac{2}{3}-4x^2+4x^3 , \quad & 0 \leq x \leq \frac{1}{2}, \\
\frac{4}{3}(1-x)^3, & \frac{1}{2} \leq x \leq1, \\
0, & \textnormal{otherwise}.
\end{cases}
\end{equation}
The correction function $C(\bm{x}; \bm{x}-\bm{y})$ in \cref{eqn:DiscUh} is defined as
\begin{equation} \label{eqn:Correction}
C(\bm{x};\bm{x-y})=\bm{H}^T(\bm{x}-\bm{y}) \bm{b}(\bm{x}),
\end{equation}
where the vector $\bm{H}^T(\bm{x}-\bm{y})$ consists of the set of monomial basis functions of order $p$,
\begin{equation} \label{eqn:monomial}
\bm{H}^T(\bm{x}-\bm{y})=[\{(\bm{x}-\bm{y})^{\bm{\alpha}} \}_{|\bm{\alpha}| \leq p}],
\end{equation}
$\bm{b}(\bm{x})$ is a vector containing correction function coefficients and can be obtained by satisfying the  $p$-th order polynomial reproduction condition,
\begin{equation}  \label{eqn:DiscRepCond}
\sum\limits_{\bm{k} \in \mathbb{Z}^{\nd}} C(\bm{x};\bm{x}-\bm{x_k}) \bm{\phi_{a}}(\bm{x}-\bm{x_k})  \bm{x_k^{\alpha}} = \bm{x^{\alpha}}, \quad |\bm{\alpha}| \leq p.
\end{equation}
Substitute \cref{eqn:Correction} into \cref{eqn:DiscRepCond} and obtain
\begin{equation}  \notag
\sum\limits_{\bm{k} \in \mathbb{Z}^{\nd}} C(\bm{x};\bm{x}-\bm{x_k}) \bm{\phi_{a}}(\bm{x}-\bm{x_k}) \bm{H}(\bm{x}-\bm{x_k})=\bm{H}(\bm{0}).
\end{equation}
Equivalently, 
\begin{equation}  \label{eqn:SOEcorrection}
\bm{M}(\bm{x}) \bm{b}(\bm{x})=\bm{H}(\bm{0}),
\end{equation}
where $\bm{M}(\bm{x})$ is the moment matrix and is formulated as
\begin{equation}  \label{eqn:eq:DiscMomentMatrix}
\bm{M}(\bm{x}) = \sum\limits_{\bm{k} \in \mathbb{Z}^{\nd}} \bm{H}(\bm{x}-\bm{x_k}) \bm{\phi_{a}}(\bm{x}-\bm{x_k})  \bm{H}^T(\bm{x}-\bm{x_k}). 
\end{equation}
Each entry of the matrix is a moment given by
\begin{equation} \label{eqn:Moment}
\bm{m_{\alpha}}(\bm{x}) = \sum_{\bm{k} \in \mathbb{Z}^{\nd}}\bm{\phi_a}(\bm{x}-\bm{x_k})(\bm{x}-\bm{x_k})^{\bm{\alpha}} =\prod^{\nd}_{j=1}m_{\alpha_j}(x_j).
\end{equation}
where $\bm{x}= (x_1,x_2,\ldots,x_\nd)$, $\bm{x}_{\bm{k}}= (x_{k_1},x_{k_2},\ldots,x_{k_\nd})$, and $m_{\alpha_j}(x_j)$ is the $\alpha_j$-th discrete moment in the $j$-th dimension given as
\begin{equation}  \label{eqn:eq:DisMoment}
m_{\alpha_j}(x_j)= \sum\limits_{k_j \in \mathbb{Z}} \phi_{a_j}(x_j- x_{k_j})(x_j-x_{k_j})^{\alpha_j} . 
\end{equation}
Solve the system of equations as in \cref{eqn:SOEcorrection} and obtain the correction function coefficients as
\begin{equation}  \label{eqn:DiscB}
\bm{b}(\bm{x})= (\bm{M}(\bm{x}))^{-1}  \bm{H}(\bm{0}).
\end{equation}
Please refer to \cite{Han2001} for the necessary conditions of the solvability of the system \cref{eqn:SOEcorrection}. 
Finally, by substituting \cref{eqn:DiscB} and \cref{eqn:Correction} into  \cref{eqn:DiscUh}, the RK approximation of $u(\bm{x})$ is obtained as
\begin{equation} \notag 
\Pi^h u(\bm{x}) = \sum\limits_{\bm{k} \in \mathbb{Z}^{\nd}} \Psi_{\bm{k}}(\bm{x})u(\bm{x_k}),
\end{equation}
where $\Psi_{\bm{k}}(\bm{x}) $ is the RK basis function, 
\begin{equation} 
\Psi_{\bm{k}}(\bm{x})=C(\bm{x};\bm{x}-\bm{x_k})\bm{\phi_{a}}(\bm{x}-\bm{x_k}) =\bm{H}^T(\bm{x}-\bm{x_k}) (\bm{M}(\bm{x}))^{-1} \bm{H}(\bm{0}) \bm{\phi_{a}}(\bm{x}-\bm{x_k}).
\end{equation}

In the rest of the work, we assume the reproducing condition \eqref{eqn:DiscRepCond} is satisfied with $p=1$,
with which we call our method the linear RK approximation and the RK basis function is referred to as the linear RK basis. 
Let $\bm{a}=2 \bm{h}$, then it can be shown (see {\it e.g.},\cite{Leng2019a}) that the correction function $C(\bm{x};\bm{x}-\bm{x_k})\equiv 1$ and the linear RK basis function is reduced to
\begin{equation}   \label{eqn:RKShape}
\Psi_{\bm{k}}(\bm{x})=\bm{\phi_a}(\bm{x}-\bm{x_k})=\prod^{\nd}_{j=1}\phi_{a_j}(x_j- x_{k_j})\,.
\end{equation}
Another consequence of this choice of support size is that the one-dimensional moments up to the third order are independent of $x_j$, and
more precisely  
\begin{equation} \label{eqn:MomentVal}
m_0(x_j) = 1, \quad 
m_1(x_j) = 0, \quad 
m_2(x_j) = \frac{h_j^2}{3}, \quad 
m_3(x_j) = 0\,, 
\end{equation}
for $j=1, \ldots, {\nd}$.
From the one-dimensional moment, we can derive useful properties of the multi-dimensional moment which are summarized in the following lemma.
\begin{lemma} \label{lem:MultiMoment}
Let $\bm{a}=2\bm{h}$, then the multi-dimensional moments satisfy the following properties,
\begin{enumerate}[label=(\roman*)]
\item $\bm{m_0} =1$ and $\bm{m}_{\bm{\alpha}}=0$ for $|\bm{\alpha}|=1 \textnormal{ or } 3 $,
\item $\bm{m_{\alpha}} = 0 \textnormal{ or } m_2(x_j) $ for $|\bm{\alpha}|=2$ and $ j =1, \ldots, {\nd} $, 
\end{enumerate}
\end{lemma}
\begin{proof}
By writing out the multi-index $\bm{\alpha}$ and from \cref{eqn:Moment} and \cref{eqn:MomentVal}, the desired properties follow.
\end{proof} 

\begin{remark}\label{rem:RKsupport}
In general, we can choose the RK support as $\bm{a}=2r_0\bm{h}$ for $r_0\in \mathbb{N}$. In this case it is shown in \cite[Lemma 4.4]{Leng2019a} that $\bm{\phi_a}$ satisfies the Strang-Fix condition (\cite{strang2011fourier}), and therefore it can be shown that the moments are constants and they satisfy the same properties in \cref{lem:MultiMoment} (\cite{Leng2019a}). In the case $p=1$, it also implies that the correction function $C(\bm{x};\bm{x}-\bm{x_k})\equiv C$ for some constant $C$. These properties are sufficient to guarantee a special synchronized convergence property (\cite{Li1996}) of RK approximation that is crucial for the consistency analysis in \cref{subsec:Convergence}. For the simplicity of presentation, we assume $r_0=1$ in this paper but the analysis also works for any $r_0 \in\mathbb{N}$.
\end{remark}

Now we use the above discussed RK approximation and collocate the nonlocal diffusion equation on the grid $\square$.
 The RK collocation scheme is formulated as follows.
Find a function $u \in S \left(\square \cap {\Omega}  \right) $ such that
\begin{equation} \label{eqn:CollocationScheme}
-\mathcal{L}_{\delta} u(\bm{x_k}) =  {f_\del}(\bm{x_k}), \, \quad \bm{x}_{\bm{k}} \in (\square \cap \Omega),
\end{equation}
where $S \left(\square \cap {\Omega} \right) $ is defined as
\begin{equation} \notag 
S \left(\square \cap {\Omega}  \right) := \left \{ u = \sum_{\bm{k} \in \mathbb{Z}^{\nd}}  \Psi_{\bm{k}} u_{\bm{k}} \, \big| \, u_{\bm{k}} = 0 \textnormal{ for such } \bm{k}  \textnormal{ that }  \bm{x_k} \notin \left(\square \cap {\Omega} \right)  \right\}.
\end{equation} 
Alternatively, \cref{eqn:CollocationScheme} can also be written as 
\begin{equation} 
- r^h_{\Omega}\mathcal{L}_{\delta} u = r^h_{\Omega} {f_\del} \, ,
\end{equation}
where $r^h_{\Omega}$ is the restriction operator given in \cref{restriction_omg}.

It is worth nothing that 
with the assumption $\bm{a}=2\bm{h}$, the RK basis has support size of $2h_j$ in the $j$-th dimension. So the support of $ u \in S \left(\square \cap {\Omega} \right)$ is not fully contained in ${\Omega}$ but in
a larger domain given as
 \[\widehat{\Omega} = (-2h_1, 1+2h_1) \times (-2h_2, 1+2h_2) \times \cdots \times (-2h_{\nd}, 1+2h_{\nd}). \]


\section{Convergence analysis of the RK collocation method} \label{sec:CAnalysis}
In this section, we will show the convergence of the RK collocation scheme \cref{eqn:CollocationScheme},
which is also the method used in \cite{Pasetto2018} without a convergence proof. 
The concern for convergence is that the numerical scheme should converge to the nonlocal problem for a fixed $\delta$,
and to the correct local problem as $\delta$ and grid size both go to zero. 
 So the proposed RK collocation scheme is an AC scheme (\cite{Tian2014}).

\subsection{Stability of the RK collocation method} \label{subsec:Stability}
 In this subsection, we provide the stability proof of our method.
 The key idea is to compare the RK collocation scheme with the Galerkin scheme using Fourier analysis. 
Similar strategies have been developed in \cite{Costabel1992}.

First, define a norm in the space of sequences by 
\begin{equation} \label{eqn:l2norm}
|(u_{\bm{k}})_{\bm{k} \in {\mathbb{Z}^{\nd}}}|_h := \| i^h(u_{\bm{k}})  \|_{L^2(\mathbb{R}^{\nd})} \,\, .
\end{equation}
If a sequence $(u_{\bm{k}})$ is only defined for $\bm{k}$ in a subset of $\mathbb{Z}^{\nd}$, then one can always use zero extension for $(u_{\bm{k}})$ so that it is defined  for all $\bm{k} \in \mathbb{Z}^{\nd}$. 
Then without further explanation,
$|(u_{\bm{k}}) |_h $  is always understood as \eqref{eqn:l2norm} 
with the zero extension being used. 
 The main theorem in this subsection is now given as follows. 
\begin{theorem} \label{thm:stability}
\textbf{(Stability I)} For any $\delta \in(0, \delta_0]$ and $u \in S(\square \cap {\Omega})$, we have
\begin{equation} \notag
\left|r^h_{\Omega}(-\mathcal{L}_{\delta} u)\right|_{h} \geq C \| u\| _{L^2(\mathbb{R}^{\nd})}\,,
\end{equation}
where $C$ is a constant that only depends on $\Omega$ and $\delta_0$.
\end{theorem}

The proof \cref{thm:stability} is shown at the end of this subsection before two more lemmas are introduced.
 Let $(\cdot \, , \cdot)_{l^2}$ be the $l^2$ norm associated inner product, namely 
\begin{equation} \notag
((u_{\bm{k}}), (v_{\bm{k}}))_{l^2} := \prod_{j=1}^{\nd}h_j \sum_{\bm{k} \in \mathbb{Z}^{\nd}}u_{\bm{k}} {v_{\bm{k}}}.
\end{equation}

\noindent
For any sequence  $(u_{\bm{k}}) \in l^2(\mathbb{Z}^{\nd})$, we define the Fourier series on $\bm{Q}=(-\pi, \pi)^{\nd}$, 
\begin{equation} \label{eqn:FourierSeries}
\widetilde{u}(\bm{\xi}) := \sum_{\bm{k} \in \mathbb{Z}^{\nd}} e^{-i\bm{k} \cdot \bm{\xi}}u_{\bm{k}},
\end{equation}
where
\begin{equation} \notag
u_{\bm{k}} = (2\pi)^{-{\nd}} \int_{\bm{Q}} e^{i\bm{k} \cdot \bm{\xi}}\widetilde{u}(\bm{\xi})d\bm{\xi}.
\end{equation}
In general $\widetilde{u}(\bm{\xi})$ is a complex-valued function and we use $\overline{\widetilde{u}(\bm{\xi})}$ to denote the complex conjugate of $\widetilde{u}(\bm{\xi})$.
The nonlocal operator $\mathcal{L}_{\delta}$ defines two discrete bilinear forms:
\begin{equation} \label{eqn:IGalerkin}
(i^h(u_{\bm{k}}), -\mathcal{L}_{\delta}i^h(v_{\bm{k}})) = \sum_{\bm{k}, \bm{k'} \in \mathbb{Z}^{\nd}}u_{\bm{k}}(\Psi_{\bm{k}}, -\mathcal{L}_{\delta}\Psi_{\bm{k'}}) {v_{\bm{k'}}},
\end{equation}
and
\begin{equation} \label{eqn:ICollocation}
((u_{\bm{k}}), -r^h\mathcal{L}_{\delta}i^h(v_{\bm{k}}))_{l^2} = \prod_{j=1}^{\nd}h_j \sum_{\bm{k}, \bm{k'} \in \mathbb{Z}^{\nd}}u_{\bm{k}}(-\mathcal{L}_{\delta}\Psi_{\bm{k'}})(\bm{x_{k}}){v_{\bm{k'}}}.
\end{equation}

\noindent 
The inner product $(\cdot \, , \cdot)$ in  \cref{eqn:IGalerkin} is the standard $L^2$ inner product.
\Cref{eqn:IGalerkin} defines a quadratic form corresponding to the Galerkin method, meanwhile, the quadratic form \cref{eqn:ICollocation} corresponds to the collocation method. Moreover, the stiffness matrix for the collocation scheme \cref{eqn:CollocationScheme} can be considered as a finite section of the infinite Toeplitz matrix induced by \cref{eqn:ICollocation}. Before applying the Fourier analysis to \cref{eqn:IGalerkin,eqn:ICollocation}, we study the Fourier symbol of the nonlocal diffusion operator $\mathcal{L}_{\delta}$ first. Take $u \in \mathcal{S}_{\delta} $, $\widehat{u}(\bm{\xi})$ is the Fourier transform of $u(\bm{x})$ defined by
\begin{equation} \notag 
\widehat{u}(\bm{\xi}) := \int_{\mathbb{R}^{\nd}} e^{-i \bm{x} \cdot \bm{\xi}} u(\bm{x}) d\bm{x}.
\end{equation}
The Fourier transform of the nonlocal diffusion operator $\mathcal{L}_{\delta}$ is given as
\begin{equation} \label{eqn:FT_diffusion}
\begin{aligned}
-\widehat{\mathcal{L}_{\delta} u}(\bm{\xi}) &=-\int_{\mathbb{R}^{\nd}} e^{-i\bm{x} \cdot \bm{\xi}} \int_{B_{\delta}(\bm{0})}\rho_{\delta}(|\bm{s}|)(u(\bm{x}+\bm{s})-u(\bm{x}))d\bm{s} d\bm{x} , \\
  &=-\int_{B_{\delta}(\bm{0})} \int_{\mathbb{R}^{\nd}} \rho_{\delta}(|\bm{s}|)(u(\bm{x}+\bm{s})-u(\bm{x})) e^{-i\bm{x} \cdot \bm{\xi}}d\bm{x} d\bm{s} , \\ 
  &=\int_{B_{\delta}(\bm{0})} \rho_{\delta}(|\bm{s}|) (1-e^{i\bm{s}\cdot\bm{\xi}}) \widehat{u}(\bm{\xi}) d\bm{s}, \\ 
  &= \lambda_{\delta}(\bm{\xi}) \widehat{u}(\bm{\xi}),
\end{aligned}
\end{equation}
where $\lambda_{\delta}(\bm{\xi})$ is the Fourier symbol of $\mathcal{L}_{\delta}$,
\begin{equation} \label{eqn:NonlocalEigen} 
\lambda_{\delta}(\bm{\xi}) =\int_{B_{\delta}(\bm{0})} \rho_{\delta}(|\bm{s}|) (1-e^{i\bm{s} \cdot \bm{\xi}})d\bm{s} =\int_{B_{\delta}(\bm{0})} \rho_{\delta}(|\bm{s}|) (1-\textnormal{cos}(\bm{s} \cdot \bm{\xi}))d\bm{s}. \\
\end{equation}
More discussions on the spectral analysis of the nonlocal diffusion operator can be found in \cite{Du2016}. From \cref{eqn:NonlocalEigen}, it is obvious that $\lambda_{\delta}(\bm{\xi})$ is real and non-negative. Now, we give a comparison of the 
two quadratic forms  \cref{eqn:IGalerkin,eqn:ICollocation} using Fourier analysis. 

\begin{lemma} \label{lem:GCF}
Let $\widetilde{u}(\bm{\xi})$ and $\widetilde{v}(\bm{\xi})$ be the Fourier series of the sequences $(u_{\bm{k}}), (v_{\bm{k}}) \in l^2(\mathbb{Z}^{\nd})$ respectively. Then
\begin{enumerate}[label=(\roman*)]
\item $(i^h(u_{\bm{k}}), -\mathcal{L}_{\delta}i^h(v_{\bm{k}})) = (2\pi)^{-{\nd}} \mathlarger{\int}_{\bm{Q}} \widetilde{u}(\bm{\xi}) \overline{\widetilde{v}(\bm{\xi})} \lambda_G(\delta, \bm{h}, \bm{\xi}) d\bm{\xi} $, \\[.05cm] \label{FGalerkin}
\item $((u_{\bm{k}}), -r^h\mathcal{L}_{\delta}i^h(v_{\bm{k}}))_{l^2} = (2\pi)^{-{\nd}} \mathlarger{\int}_{\bm{Q}} \widetilde{u}(\bm{\xi}) \overline{\widetilde{v}(\bm{\xi})} \lambda_C(\delta, \bm{h}, \bm{\xi}) d\bm{\xi} $, \\[.05cm] \label{FCollocation}
\item $\lambda_C(\delta, \bm{h}, \bm{\xi}) \geq  C \lambda_G(\delta, \bm{h}, \bm{\xi})$, for $C$ independent of $\delta, \bm{h}$ and $\bm{\xi}$, \label{FGCEquivalent}
\end{enumerate}
and $\lambda_G$ and $\lambda_C$ are given by
\begin{equation} \label{eqn:lambdaG}
\lambda_G(\delta, \bm{h}, \bm{\xi}) =2^{8{\nd}}  \sum_{\bm{r} \in \mathbb{Z}^{\nd}} \lambda_{\delta} \left((\bm{\xi} + 2 \pi \bm{r}) \oslash{\bm{h}} \right) \prod_{j=1}^{\nd} h_j \left(\frac{\textnormal{sin}(\xi_j /2)}{\xi_j + 2 \pi r_j} \right)^{8}, 
\end{equation}
\begin{equation} \label{eqn:lambdaC}
\lambda_C(\delta, \bm{h}, \bm{\xi}) = 2^{4{\nd}} \sum_{\bm{r} \in \mathbb{Z}^{\nd}} \lambda_{\delta}\left((\bm{\xi} + 2 \pi \bm{r}) \oslash{\bm{h}} \right) \prod_{j=1}^{\nd} h_j \left(\frac{\textnormal{sin}(\xi_j /2)}{\xi_j + 2 \pi r_j} \right)^{4}. 
\end{equation}
\end{lemma}

\noindent 
\begin{proof} 
Using \cref{eqn:FT_diffusion} and Parseval's identity, we arrive at
\begin{equation} \notag
\begin{aligned}
&(\Psi_{\bm{k}}, -\mathcal{L}_{\delta}\Psi_{\bm{k'}}) = (2\pi)^{-{\nd}} \int_{\mathbb{R}^{\nd}} \widehat{\Psi_{\bm{k}}}(\bm{\xi}) \overline{\lambda_{\delta}(\bm{\xi})} \overline{\widehat{\Psi_{\bm{k'}}}(\bm{\xi})} d\bm{\xi}, \\
  &= (2\pi)^{-{\nd}} \int_{\mathbb{R}^{\nd}} e^{i(\bm{x_{k'}}-\bm{x_k}) \cdot \bm{\xi}} \lambda_{\delta}(\bm{\xi}) |\widehat{\Psi_{\bm{0}}}|^2 (\bm{\xi}) d\bm{\xi},\\
  &= (2\pi)^{-{\nd}} \sum_{\bm{r}\in \Z^\nd}\int_{\bm{Q}} \left( e^{i(\bm{k'}-\bm{k})\cdot \bm{\xi}} \lambda_\del((\bm{\xi} + 2 \pi \bm{r}) \oslash \bm{h}) |\widehat{\Psi_{\bm{0}}}|^2 ((\bm{\xi} + 2 \pi \bm{r}) \oslash \bm{h}) d\bm{\xi} \right)/ \prod_{j=1}^{\nd} h_j\,. 
\end{aligned}
\end{equation}
Using \cref{eqn:RKShape} and the Fourier transform of the cubic B-spline function \cref{eqn:CubicSpline} given by 
\begin{equation} \notag
\widehat{\phi}(\xi) = \frac{1}{2}\left(\frac{\textnormal{sin}(\xi/4)}{\xi/4}\right)^4,
\end{equation}
we have the Fourier transform of the RK basis function
\begin{equation} \notag
\widehat{\Psi_{\bm{0}}}(\bm{\xi})= \prod_{j=1}^{\nd} \widehat{\phi \left(\frac{x_j}{2h_j}\right)}(\xi_j) = \prod_{j=1}^{\nd} h_j\left(\frac{\textnormal{sin}(h_j\xi_j/2)}{(h_j\xi_j/2)}\right)^4.
\end{equation}
\noindent
Therefore,
\[
(\Psi_{\bm{k}}, -\mathcal{L}_{\delta}\Psi_{\bm{k'}})= 
(2\pi)^{-{\nd}}\int_{\bm{Q}} e^{i(\bm{k'}-\bm{k})\cdot \bm{\xi}} \lambda_G(\delta, \bm{h}, \bm{\xi}) d\bm{\xi}\,,
\]
where $\lambda_G$ is given by \cref{eqn:lambdaG}. 
Combing the above equation with \cref{eqn:FourierSeries} and \cref{eqn:IGalerkin}, we obtain {\it (i)},
\begin{equation} \notag
\begin{aligned}
(i^h(u_{\bm{k}}), -\mathcal{L}_{\delta}i^h(v_{\bm{k}})) &= (2\pi)^{-{\nd}}\sum_{\bm{k}, \bm{k'} \in \mathbb{Z}^{\nd}}u_{\bm{k}}{v_{\bm{k'}}} \int_{\bm{Q}} e^{i(\bm{k'}-\bm{k}) \cdot \bm{\xi}} \lambda_G(\delta, \bm{h}, \bm{\xi}) d\bm{\xi}, \\
                &= (2\pi)^{-{\nd}}\int_{\bm{Q}} \widetilde{u}(\bm{\xi}) \overline{\widetilde{v}(\bm{\xi})} \lambda_G(\delta, \bm{h}, \bm{\xi}) d\bm{\xi}\,. \\
\end{aligned}
\end{equation}

Next, following the same procedure, we arrive at the collocation matrix expressed as
\begin{equation} \notag
\begin{aligned}
-\mathcal{L}_{\delta}\Psi_{\bm{k'}}(\bm{x_k}) &= (2\pi)^{-{\nd}} \int_{\mathbb{R}^d} e^{i \bm{x_k} \cdot \bm{\xi}}\lambda_{\delta}(\bm{\xi}) \widehat{\Psi_{\bm{k'}}}(\bm{\xi}) d\bm{\xi}, \\
  &= (2\pi)^{-{\nd}} \int_{\mathbb{R}^{\nd}} e^{i (\bm{x_k}-\bm{x_{k'}}) \cdot \bm{\xi}} \lambda_{\delta}(\bm{\xi}) \widehat{\Psi_{\bm{0}}}(\bm{\xi}) d\bm{\xi},\\
  &=  (2\pi)^{-{\nd}}\int_{\bm{Q}} e^{i(\bm{k}-\bm{k'}) \cdot \bm{\xi}} \lambda_C(\delta, \bm{h}, \bm{\xi}) d\bm{\xi}\,. \\
\end{aligned}
\end{equation}
Therefore, the collocation form \cref{eqn:ICollocation} is written as
\begin{equation} \notag
\begin{aligned}
((u_{\bm{k}}), -r^h\mathcal{L}_{\delta}i^h(v_{\bm{k}}))_{l^2} &= (2\pi)^{-{\nd}}\sum_{\bm{k}, \bm{k'} \in \mathbb{Z}^{\nd}}u_{\bm{k}}{v_{\bm{k'}}} \int_{\bm{Q}} e^{i(\bm{k'}-\bm{k})\cdot \bm{\xi}} \lambda_C(\delta, \bm{h}, \bm{\xi}) d\bm{\xi}, \\
  &= (2\pi)^{-{\nd}}\int_{\bm{Q}} \widetilde{u}(\bm{\xi}) \overline{\widetilde{v}(\bm{\xi})} \lambda_C(\delta, \bm{h}, \bm{\xi}) d\bm{\xi}. \quad   \\
\end{aligned}
\end{equation}
This finishes the proof of \ref{FCollocation}.

With \ref{FGalerkin} and \ref{FCollocation} being established, it is easy to see \ref{FGCEquivalent} by the fact that $\lambda_{\delta}(\bm{\xi})$ is non-negative and $0 \leq |\sin(x)/x| \leq 1$. \end{proof}


Before showing the proof of  \cref{thm:stability}, we need the following lemma that says the $|\cdot|_h$ norm defined through the RK basis and 
 discrete $l^2$ norm are equivalent.
\begin{lemma} \label{lem:L2l2}
The following two norms are equivalent, i.e., there exist two constants $C_1, C_2 > 0$ independent of $h$, such that
\begin{equation} \notag
C_1  \| u \|_{l^2(\mathbb{R}^{\nd})} \leq |(u_{\bm{k}})_{\bm{k} \in \mathbb{Z}^{\nd}}|_h \leq C_2 \| u \|_{l^2(\mathbb{R}^{\nd})}.
\end{equation}
\end{lemma}
\noindent 
\begin{proof}
First, from Parsevel's identity, we can write the $l^2$ norm as 
\begin{equation}  \notag 
\begin{aligned}
\| u \|^2_{l^2(\mathbb{R}^{\nd})} &=  \prod_{j=1}^{\nd} h_j \sum_{\bm{k} \in \mathbb{Z}^d}| u_{\bm{k}}|^2 , \\
		&= (2\pi)^{-{\nd}} \prod_{j=1}^{\nd} h_j \int_{\bm{Q}}  \widetilde{u}^2(\bm{\xi}) d\bm{\xi}.
\end{aligned}
\end{equation}
Then, similar to the proof of \cref{lem:GCF} \ref{FGalerkin}, by replacing the nonlocal diffusion operator with the identity operator, we obtain
\begin{equation}  \notag 
|(u_{\bm{k}})_{\bm{k} \in \mathbb{Z}^{\nd}}|_h^2 = (i^h(u_{\bm{k}}), i^h(u_{\bm{k}}))=  (2\pi)^{-{\nd}} \prod_{j=1}^{\nd} h_j  \int_{\bm{Q}} \beta(\bm{\xi}) \widetilde{u}^2(\bm{\xi})  d\bm{\xi} ,
\end{equation}
where $\beta(\bm{\xi}) $ is continuous and strictly positive,
\begin{equation} \notag
\beta(\bm{\xi}) = 2^{8{\nd}} \sum_{\bm{r} \in \mathbb{Z}^{\nd}} \prod_{j=1}^{\nd} \left(\frac{\textnormal{sin}(\xi_j /2)}{\xi_j + 2 \pi r_j} \right)^{8}.
\end{equation}
Thus, $\beta(\bm{\xi})$ is bounded above and below on $\bm{Q}$. Therefore, we complete the proof.
\end{proof} 

\noindent 
\begin{proof}[Proof of of  \cref{thm:stability}]
For all sequences $(u_{\bm{k}}), (v_{\bm{k}})$, 
we derive via the Cauchy-Schwartz inequality and \cref{lem:L2l2} 
\begin{equation} \label{eqn:cauchyschwartz}
\begin{aligned}
|((u_{\bm{k}}), (v_{\bm{k}}))_{l^2}| &= \prod_{j=1}^{\nd}h_j  \left| \sum_{\bm{k} \in \mathbb{Z}^{\nd}}u_{\bm{k}} {v_{\bm{k}}} \right|, \\
                   & \leq  \left(\prod_{j=1}^{\nd}h_j \sum_{\bm{k} \in \mathbb{Z}^{\nd}} |u_{\bm{k}}|^2 \right)^{1/2}  \left(\prod_{j=1}^{\nd}h_j  \sum_{\bm{k} \in \mathbb{Z}^{\nd}} |v_{\bm{k}}|^2 \right)^{1/2}, \\
                 &\leq C |(u_{\bm{k}})|_h \cdot |(v_{\bm{k}})|_h .
\end{aligned}
\end{equation}
Finally, for $u \in S \left(\square \cap {\Omega} \right) $ we may by definition write $u = i^h(u_{\bm{k}})$, and thus we have 
\begin{align*}
\left| (u_{\bm{k}})\right|_h \cdot  \left|r^h_{\Omega}(-\mathcal{L}_{\delta} u) \right|_h &\geq C \left|( (u_{\bm{k}}),  r^h_{\Omega}(-\mathcal{L}_{\delta} u))_{l^2} \right|,  \\
   &= C \left|((u_{\bm{k}}),  r^h (-\mathcal{L}_{\delta} i^h(u_{\bm{k}})))_{l^2} \right|,  \\
   &\geq C \left|(i^h(u_{\bm{k}}), (-\mathcal{L}_{\delta} i^h(u_{\bm{k}}))) \right|,  \\
   &\geq C \| u\|^2_{L^2(\mathbb{R}^{\nd})} \,.
\end{align*}
The first line is a result of \cref{eqn:cauchyschwartz} and the second line is by definition of $S(\square \cap {\Omega})$. \cref{lem:GCF} \ref{FGCEquivalent} shows the third line and the fourth line is from the stability given by \cref{thm:nonlocalmapping} since $u\in S \left(\square \cap {\Omega} \right)\subset \mathcal{S}_\delta$ and for a sufficiently large and fixed domain $\widetilde{\Om}\supset \Om$ we have $u|_{\R^\nd \backslash\widetilde{\Om}}=0$. 
\end{proof}

\subsection{Consistency of the RK collocation method} \label{subsec:Convergence}
In this section, we discuss uniform consistency of the RK collocation method on the nonlocal diffusion models, namely that  truncation error is independent of the nonlocal scaling parameter $\delta$. The uniform consistency result is crucial to show the asymptotic compatibility of the scheme.  Combining the stability result in \cref{subsec:Stability} and the truncation error analysis to be presented shortly, we show that the RK collocation method is convergent. The numerical solution is convergent to the nonlocal solution with a fixed nonlocal parameter $\delta$ (\cref{thm:NonlocalConvergenceTHM}) and to the corresponding local limit as $\delta$ and the mesh spacing both go to zero (\cref{thm:AC}).   
If the RK support size is carefully chosen, RK approximation has the synchronized convergence property \cite{Li1996}, which is the key ingredient to show the uniform consistency of the collocation scheme \cref{eqn:CollocationScheme}.

\cite[Theorem 5.2]{Li1996} shows that the
 synchronized convergence property holds if the kernel function $\bm{\phi_{a}}$ defined in \cref{eqn:CartesianWindow} satisfies the Strang-Fix condition and the correction function $C(\bm{x};\bm{x}-\bm{x_k})$ defined in \cref{eqn:Correction} is a constant (the original work assumes the constant is $1$ but it is easy to see the result also holds with any constant because the set of functions satisfying the Strang-Fix condition is invariant under a constant multiplication). Since we see in \cref{rem:RKsupport} that the special choice of the support $\bm{a}=2r_0\bm{h}$ ($r_0\in \mathbb{N}$) implies that $\bm{\phi_{a}}$ satisfies the Strang-Fix condition and  $C(\bm{x};\bm{x}-\bm{x_k})\equiv C$, the synchronized convergence property is guaranteed. In \cite{Li1996}, the RK approximation errors are measured in Sobolev norms, but its proof also shows that point-wise errors are controlled under stronger regularity assumptions of the approximated functions. Here we present the result without proof and the readers are referred to \cite{Li1996,Li1998synchronized} for more details.

\begin{lemma} \label{lem:synchronizedconvergence}
\textbf{(Synchronized Convergence)}
Assume $u(\bm{x}) \in C^4(\mathbb{R}^{\nd}) $ and  $\Pi^h u$  is the RK interpolation with
the shape function given by \cref{eqn:RKShape}. $\Pi^h u$ has synchronized convergence, namely 
\begin{equation} \notag 
\left|D^{\bm{\alpha}}(\Pi^h u - u )  \right|_\infty \leq C |u^{(|\bm{\alpha}| + 2)} |_{\infty} h_{\textnormal{max}}^2, \quad \textnormal{for } |\bm{\alpha}| = 0, 1, 2,
\end{equation}
where $C$ is a generic constant independent of $h_{\textnormal{max}}$\,.
\end{lemma}
Here and in the rest of the paper, we adopt the following notations for a function 
$u\in C^n(\R^\nd)$,
\begin{equation} \notag 
|u|_\infty = \sup_{\bm{x}\in \R^\nd} |u(\bm{x})|, \text{ and } |u^{(l)}|_{\infty} = \sup_{|\bm{\beta}|=l}\sup_{\bm{y} \in \mathbb{R}^{\nd}}|D^{\bm{\beta}} u(\bm{y})| \;(1\leq l \leq n).
\end{equation} 
Now we are ready to present the truncation error analysis of the RK collocation method for the nonlocal diffusion models.
\begin{lemma} \label{lem:consistency}
\textbf{(Uniform consistency)} Assume $u(\bm{x}) \in C^4(\R^{\nd})$, then
\begin{equation} \notag 
\left| r^h \mathcal{L}_{\delta}\Pi^hu - r^h \mathcal{L}_{\delta} u \right|_{h} \leq C h_{\max}^2 |u^{(4)}|_{\infty},
\end{equation}
 where $C$ is independent of $h_{\max}$ and $\delta$.  
\end{lemma}
\begin{proof}
Define the interpolation error of $u(\bm{x})$ as
\begin{equation} \notag 
E(\bm{x}) = \Pi^h u(\bm{x}) - u(\bm{x}). 
\end{equation}
Restricting to the grid  $ \square $, the truncation error is given by 
\begin{equation} \label{eqn:TrunError}
\begin{aligned}
\left| \mathcal{L}_{\delta} \left(\Pi^h u - u\right) (\bm{x_k})  \right| 
&=  \left|  \int_{B_{\delta}(\bm{0})} \rho_{\delta}(|\bm{s}|) \left( E(\bm{x_k}+\bm{s}) - E(\bm{x_k}) \right) d\bm{s} \right|\,.
\end{aligned}
\end{equation}
Now using  \cref{lem:synchronizedconvergence} on $E = \Pi^h u - u$, we have 
\begin{equation} \label{eqn:RemainderError}
\begin{aligned}
\left| E(\bm{x_k}+\bm{s}) + E(\bm{x_k}-\bm{s}) - 2E(\bm{x}) \right| 
     & \leq  C |\bm{s}|^2 \max_{|\bm{\alpha}|=2}  \big|D^{\bm{\alpha}}E \big|_{\infty}  \\
      & \leq  C |\bm{s}|^2 |u^{(4)}|_{\infty} h^2_{\textnormal{max}}\,.  \\
\end{aligned}
\end{equation}
Now combing \cref{eqn:TrunError},\cref{eqn:RemainderError}, and \cref{eqn:boundedmoment},
we arrive at 
\[\left| \mathcal{L}_{\delta} \left(\Pi^h u - u\right) (\bm{x_k})  \right|  
  \leq C h_{\textnormal{max}}^2 | u^{(4)} |_{\infty} \,. 
\]
where $C$ is a generic constant, independent of $h_{\max}$ and $\delta$.
Finally, the proof is finished by interpolating the truncation error. 
\end{proof}

\begin{remark} \label{rem:QuadraticExactness}
In \cite{Yang2018}, the authors consider finite difference schemes for nonlocal diffusion models and it is shown that the key to obtain uniform truncation error independent of the nonlocal parameter $\delta$ is the {\it quadratic exactness} of the scheme. That is, the numerical approximation to the nonlocal diffusion operator is exact for quadratic polynomials. We remark that  although linear RK approximation using shape function \cref{eqn:RKShape} can only reproduce multilinear polynomials,  it shifts quadratic polynomials by a constant, {\it i.e.,}
\begin{equation} \label{eqn:quadraticapproximation}
\begin{aligned}
\sum_{|\bm{\alpha}|=2} \Pi^h \bm{x}^{\bm{\alpha}}
 & = \sum_{|\bm{\alpha}|=2} \left[ \bm{m}_{\bm{\alpha}}(\bm{x}) + \bm{x}^{\bm{\alpha}} \right]\,,
\end{aligned}
\end{equation}
where $\bm{m}_{\bm{\alpha}}(\bm{x})\equiv C$ as a result of \cref{lem:MultiMoment} and \cref{eqn:MomentVal}. 
Therefore the quadratic exactness is satisfied, namely for $u(\bm{x})=\bm{x}^{\bm{\alpha}}$, $|\bm{\alpha}| = 2$, we have 
\begin{equation} \notag 
\sum_{|\bm{\alpha}|=2} \mathcal{L}_{\delta} \Pi^h u(\bm{x}) = \sum_{|\bm{\alpha}|=2} \mathcal{L}_{\delta} u(\bm{x}). 
\end{equation}

\end{remark}

The convergence theorem is now presented as a result of the stability (\cref{thm:stability})  and consistency (\cref{lem:consistency}). We will show first that the numerical solution converges to the nonlocal solution for fixed $\delta$ as mesh size decreases, 
and then the convergence to the local solution as $\delta$ and mesh size both decrease to zero. 
\begin{theorem} \label{thm:NonlocalConvergenceTHM}
\textbf{(Uniform Convergence to nonlocal solution)} For a fixed $\delta \in (0, \delta_0]$, assume the nonlocal exact solution $u^{\delta}$ is sufficiently smooth, i.e., $u^{\delta} \in C^4(\overline{\Omega_{\mathcal{\delta}}})$. Moreover, assume $|{u^{\delta}}^{(4)}|_{\infty}$ is uniformly bounded for every $\delta$. Let $u^{\delta, h}$ be the numerical solution of the collocation scheme \cref{eqn:CollocationScheme}. Then,
\begin{equation} \notag 
\| u^{\delta} - u^{\delta, h}  \|_{L^2(\Omega)} \leq C h_{\max}^2,
\end{equation}
where $C$ is independent of $h_{\max}$ and $\delta$.
\end{theorem}
\begin{proof}
Notice that since $u^{\delta} =0 $ on $\Omega_{\cI}$ and $u^\delta \in C^4 (\overline{\Omega\cup\Omega_\cI})$,
we can extend $u^{\delta}$ to $\R^\nd$ by zero such that $u^\delta \in C^4(\R^{\nd})$. 
From the RK collocation scheme \cref{eqn:CollocationScheme} and the nonlocal equation \cref{eqn:NonlocalEqn}, we have
\[ 
- r^h_{\Omega}\mathcal{L}_{\delta} u^{\delta,h} = r^h_{\Omega} {f_\del} = - r^h_{\Omega}\mathcal{L}_{\delta} u^{\delta} \,.
\]
Combining  \cref{thm:stability},  \cref{lem:consistency} and the above equation, we obtain
\begin{equation} \notag 
\begin{aligned}
\| \Pi^h u^{\delta}  - u^{\delta,h}  \|_{L^2(\R^\nd)} &    \leq  C \left| r^h_{\Omega} \mathcal{L}_{\delta} \left(\Pi^h u^{\delta} - u^{\delta,h}  \right)  \right|_{h}, \\
& \leq C \left| r^h_{\Omega} \mathcal{L}_{\delta} \Pi^h u^{\delta} - r^h_{\Omega} \mathcal{L}_{\delta} u^{\delta, h}  \right|_{h},\\
& \leq C \left| r^h_{\Omega} \mathcal{L}_{\delta} \Pi^h u^{\delta} - r^h_{\Omega} \mathcal{L}_{\delta} u^{\delta}  \right|_{h},\\
& \leq C  h_{\max}^2 \, .
\end{aligned}
\end{equation}

Finally, from the triangle inequality, we arrive at
\begin{equation} \notag 
\begin{aligned}    
\|u^{\delta}  - u^{\delta,h}  \|_{L^2(\R^\nd)} \leq \| u^{\delta}  - \Pi^h u^{\delta}  \|_{L^2(\R^\nd)}  + \| \Pi^h u^{\delta}  - u^{\delta,h}  \|_{L^2(\R^\nd)} 
 \leq C h^2_{\max}.
\end{aligned}
\end{equation}
where we have used the approximation property of the RK approximation. 
\end{proof}

Next, we show that the RK collocation scheme converges to the correct local limit model. We start by bounding the truncation error between the collocation scheme and local limit of the nonlocal model. 
\begin{lemma} \label{lem:DiscreteME}
\textbf{(Discrete model error I)} Assume $u(\bm{x}) \in C^4(\R^{\nd})$, then
\begin{equation} \notag 
\left| r^h \mathcal{L}_{\delta}\Pi^hu - r^h \mathcal{L}_0 u \right|_h \leq C |u^{(4)}|_{\infty} (h_{\max}^2 + \delta^2) ,
\end{equation}
where $C$ is independent of $h_{\max}$ and $\delta$. 
\end{lemma}
\begin{proof}
From \cref{lem:consistency} and the continuum property of the nonlocal operators, we have
\begin{align*} 
\left| r^h \mathcal{L}_{\delta} \Pi^hu - r^h \mathcal{L}_{0}  u \right|_{h} & \leq  \left|  r^h \mathcal{L}_{\delta} \Pi^hu -  r^h \mathcal{L}_{\delta} u \right|_{h}  + \left|  r^h \mathcal{L}_{\delta} u -  r^h \mathcal{L}_{0} u \right|_{h}, \\ 
& \leq C |u^{(4)}|_{\infty} (h_{\max}^2 +  \delta^2),, 
\end{align*}
\end{proof}

Combining  \cref{thm:stability} and \cref{lem:DiscreteME}, we have uniform convergence (asymptotic compatibility) to the local limit. 
\begin{theorem} \textbf{(Asymptotic compatibility)} \label{thm:AC}
Assume the local exact solution $u^0$ is sufficiently smooth, i.e., $u^0 \in C^4(\overline{ \Omega_{\delta_0}})$. For any $\delta \in (0,  \delta_0]$, $u^{\delta, h}$ is the numerical solution of the collocation scheme \cref{eqn:CollocationScheme}, then,
\begin{equation} \notag 
\| u^0 - u^{\delta, h} \|_{L^2(\Omega)} \leq C (h_{\max}^2 + \delta^2).
\end{equation}
\end{theorem}
\begin{proof} 
 First, recall that
\[ 
- r^h_{\Omega}\mathcal{L}_{\delta} u^{\delta,h} = r^h_{\Omega} {f_\del} \quad \text{and} \quad - r^h_{\Omega}\mathcal{L}_{0} u^{0}=r^h_{\Omega} {f_0}  \, ,
\]
where $\max_{\bm{x}}|f_\del(\bm{x}) -f_0(\bm{x})|=O(\del^2)$.
Then, from \cref{thm:stability}, \cref{lem:DiscreteME}, we obtain
\begin{align*}
\| \Pi^h u^{0}  - u^{\delta,h}  \|_{L^2(\R^\nd)} &    \leq  C \left| r^h_{\Omega} \mathcal{L}_{\delta} \left(\Pi^h u^{0} - u^{\delta,h}  \right)  \right|_{h} ,\\
& \leq C \left| r^h_{\Omega} \mathcal{L}_{\delta} \Pi^h u^{0} - r^h_{\Omega} \mathcal{L}_{\delta} u^{\delta, h}  \right|_{h},\\
& \leq C \left| r^h_{\Omega} \mathcal{L}_{\delta} \Pi^h u^{0} - r^h_{\Omega} \mathcal{L}_{0} u^{0} + r^h_{\Omega} f_\del -   r^h_{\Omega} f_0 \right|_{h},\\
& \leq C ( h_{\max}^2 +\del^2) \, .
\end{align*}

Finally, we finish the proof by applying the triangle inequality
\begin{equation*} 
\|u^{0}  - u^{\delta,h}  \|_{L^2(\R^\nd)} \leq \| u^{0}  - \Pi^h u^{0}  \|_{L^2(\R^\nd)}  +  \| \Pi^h u^{0}  - u^{\delta,h}  \|_{L^2(\R^\nd)} \leq C (h_{\max}^2 + \delta^2).
\end{equation*}
\,
\end{proof}

\section{Quasi-discrete nonlocal diffusion operator} \label{sec:QuasiDiscrete}

The RK collocation scheme introduced in the previous sections is asymptotically compatible, but it is not practical in the sense that it is rather difficult to evaluate the integral in the nonlocal diffusion operator, especially if the nonlocal kernel is singular. In \cite{Pasetto2018}, two Gauss quadrature schemes are investigated and high-order Gauss quadrature rules are necessary for both schemes to obtain algebraic convergence.
See also \cref{sec:NumericalExample} for more details on the Gauss quadrature schemes. 
To mitigate this computational complexity, in this section we introduce a new nonlocal diffusion operator acting on continuous function where the integral is replaced by finite summation of point evaluations in the horizon. We call it the quasi-discrete nonlocal diffusion operator. Whenever the local limit is concerned, it is much easier to use the quasi-discrete operator than to use Gauss quadrature for the integral. We will show that the numerical solution for the quasi-discrete nonlocal diffusion converges to the solution of local equation as $\delta$ and mesh size both approach zero.

\subsection{Quasi-discrete nonlocal diffusion operator} \label{subsec:QDND}
For each $\bm{x}$, we use a finite number of quadrature points in the $\delta$-neighborhood of $\bm{x}$ to approximation the integral in \cref{eqn:NonlocalOper}. Assume $u(\bm{x}) \in C^0(\Omega_\delta)$, we define the quasi-discrete nonlocal diffusion operator $\mathcal{L}^{\epsilon}_{\delta}$ as
\begin{equation} \label{eqn:DiscreteNonlocalOperator}
\mathcal{L}^{\epsilon}_{\delta}u(\bm{x}) = 2\sum\limits_{\bm{s} \in B^{\epsilon}_{\delta}(\bm{0})}  \omega_{\delta}(\bm{s}) \rho_{\delta}(|\bm{s}|) (u(\bm{x}+\bm{s})-u(\bm{x})), \quad \forall \, \bm{x} \in \Omega 
\end{equation} 
where $\omega_{\delta}(\bm{s})$ is the quadrature weight at the quadrature point $\bm{s}$ and $ B^{\epsilon}_{\delta}(\bm{0}) $ is a finite collection of symmetric quadrature points $\bm{s}$ in the ball of radius $\delta$ about $\bm{0}$. 
The superscript $\epsilon$ indicates the spacing of the quadrature points (see \cref{fig:DiscreteBall} as an example). 
 $\epsilon$ is independent of the spatial discretization applied to $u$ later in \cref{sec:DCAnalysis}.
In this work, we assume that the number of quadrature points, $N_{\nd}$, in  $B^{\epsilon}_{\delta}(\bm{0}) $ is fixed  and  only depends on the dimension, $\nd$. There are some restrictions on the quadrature points. First we require that quadrature points being symmetrically distributed so that if $ \bm{x} = (x_1, \ldots, x_\nd ) \in  B^{\epsilon}_{\delta}(\bm{0})$, then $-\bm{e}_j \odot \bm{x}\in B^{\epsilon}_{\delta}(\bm{0}) $ ($\bm{e}_j$ is the unit vector with $j$-th component to be $1$) for any $j\in \{1,\ldots, \nd\}$, and any reordering of $(x_1, \ldots, x_\nd )$ also belongs to $B^{\epsilon}_{\delta}(\bm{0})$. This assumption leads to $N_\nd \geq 2\nd$ and it guarantees positiveness of the weights which will be explained in \cref{subsec:QW_RK}.  Secondly, we require that $N_\nd\geq 4\nd$ with the reasons to be explained in the proof of \cref{thm:DStability}. If we use a set of uniform distributed quadrature points as shown in \cref{fig:DiscreteBall}, then a fixed  number $N_\nd$ implies that the ratio $\delta/\epsilon$ is a fixed number.  

\begin{figure}[H]
\centering
\scalebox{1}{\includegraphics{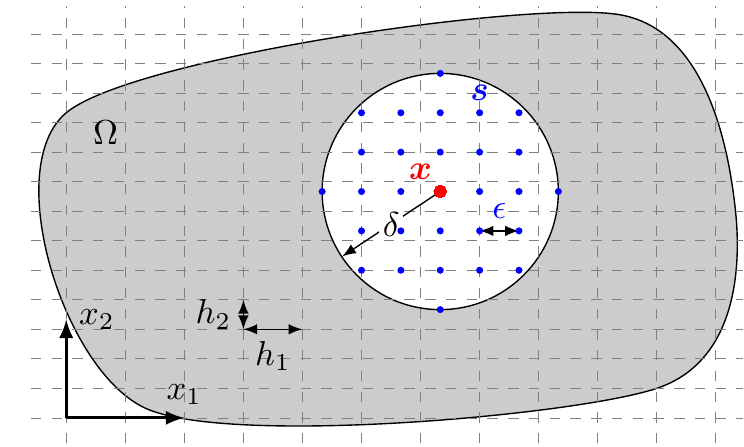}}
\caption{An example of quadrature points (blue dots) distributed in the $\del$-neighborhood of a given point $\bm{x}\in\Omega$ (red dot). $\ep$ is the spacing of the quadrature points. The dashed lines form the grid for the RK collocation scheme  that will be introduced in \cref{sec:DCAnalysis}. Notice that the quadrature points are independent of the grid points. In this example, $\delta =3 \epsilon$.}
\label{fig:DiscreteBall}
\end{figure}

Note that although the evaluation of $\mathcal{L}^{\epsilon}_{\delta}(\bm{x})$ only needs a finite summation, it is actually a continuous function in $\bm{x}$, and thus it is called the quasi-discrete operator.  Next we impose the conditions on  $\mathcal{L}^{\epsilon}_{\delta}$ so that the numerical solution  converges to the correct local limit. Recall the nonlocal kernel $\rho_{\delta}(|\bm{s}|)$ introduced in \cref{sec:NonlocalEquation} has a bounded second-order moment as in \cref{eqn:boundedmoment}. Therefore, the nonlocal diffusion operator $\mathcal{L}_{\delta}$ acting on constant, linear and quadratic polynomials has the following results
\begin{equation} \label{eqn:CQuadMoment}
\mathcal{L}_{\delta} \bm{x}^0 = 0; \quad
\mathcal{L}_{\delta} \bm{x}^{\bm{\alpha}} = 0, \textnormal{ for } |\bm{\alpha}|=1; \quad 
\sum_{|\bm{\alpha}|=2}\mathcal{L}_{\delta} \bm{x}^{\bm{\alpha}} = 2{\nd}.
\end{equation}
We proceed to design quadrature weights $\omega_{\delta}(\bm{s})$ such that $\mathcal{L}^{\epsilon}_{\delta}$ satisfies the same conditions as \cref{eqn:CQuadMoment}, 
\begin{equation} \label{eqn:DQuadMoment}
\mathcal{L}^{\epsilon}_{\delta} \bm{x}^0 = 0; \quad
\mathcal{L}^{\epsilon}_{\delta} \bm{x}^{\bm{\alpha}} = 0, \textnormal{ for } |\bm{\alpha}|=1; \quad 
\sum_{|\bm{\alpha}|=2}\mathcal{L}^{\epsilon}_{\delta} \bm{x}^{\bm{\alpha}} = 2{\nd}.
\end{equation}
There are many ways to design the quadrature weights $\omega_{\delta}(\bm{s})$. We present two approaches, one is using the RK approximation \cite{Pasetto2019} and the other is in the framework of GMLS as in \cite{Trask2019}. We modify the RK approach and emphasize the positivity of the weights in the following subsection such that the convergence analysis shown in \cref{sec:DCAnalysis} works for both RK approach and GMLS framework. 

\subsection{Quadrature weights using the RK approximation} \label{subsec:QW_RK}
Note that
 \cref{eqn:DQuadMoment} can be seen as a reproducing condition for polynomials up to second order. 
 We can thus solve for the quadrature weights $\omega_{\delta}$ under the same framework presented in \cref{sec:RKCollocation}. Due to the scaling of the nonlocal kernel $\rho_{\delta}(|\bm{s}|)$ as shown in \cref{eqn:NonlocalKernelScaling}, $\omega_{\delta}(\bm{s})$ has the following scaling
\begin{equation} \notag
\omega_{\delta}(\bm{s}) =\delta^{\nd} \omega \left( \frac{\bm{s}}{\delta} \right),
\end{equation}
where $\omega(\bm{s})$ is the quadrature weight at $\bm{s} \in B^{\epsilon_1}_1(\bm{0})$ and $\epsilon_1 =\epsilon/\delta$ by scaling. $B^{\epsilon_1}_1(\bm{0})$ is a finite collection of quadrature points in the unit ball.  Since we assume that the number of points in the set $B^\epsilon_\delta(\bm{0})$ is fixed, this implies $\epsilon_1$ is a fixed number. We next present the construction of $\omega(\bm{s})$ in $B^{\epsilon_1}_1(\bm{0})$ so that for a general horizon $\delta$, $\omega_{\delta}(\bm{s})$ can be obtained by rescaling. 
Without loss of generality, we assume $\bm{x}=\bm{0}$ in \cref{eqn:DiscreteNonlocalOperator}. 
If we define
 \begin{equation} \label{eqn:defn_f}
  f(\bm{s})= \rho(\bm{s}) (u( \delta\bm{s})- u(\bm{0}))
 \end{equation}
  the problem is then equivalent to finding out the weights $\omega(\bm{s})$ such that
\begin{equation} \label{eqn:rep_cond_general}
\sum_{\bm{s}\in B_1^{\epsilon_1}(\bm{0})} \omega(\bm{s})f(\bm{s}) d\bm{s}  = \int_{B_1(\bm{0})} f(\bm{s}) d\bm{s}\,,
\end{equation}
where $f$ is taken from a certain finite dimensional space. 
Condition \cref{eqn:DQuadMoment} is now interpreted as  \cref{eqn:rep_cond_general} for $f$ in the form 
of \cref{eqn:defn_f} where $u$ is taken from the space of polynomials up to second order. 
Notice that for $u$ being a constant function,  $f$ is identically zero and \cref{eqn:rep_cond_general} is satisfied trivially. Therefore 
we only need to consider \cref{eqn:rep_cond_general}  for 
 \begin{equation}\label{eqn:form_of_f}
  f(\bm{s})= \rho(\bm{s}) V(\bm{s}), \text{ where }V(\bm{s})= \bm{s}^{\bm{\alpha}}  , 
 1\leq |\bm{\alpha} |\leq 2. 
 \end{equation}

Now we use a similar approach as the RK approximation presented in \cref{sec:RKCollocation} to
find the weights. Define 
\begin{equation} \label{eqn:QuadratureWeight} 
\omega_{}(\bm{s}) = C(\bm{s}) \phi_{}(|\bm{s}|),
\end{equation}
where $\phi$ is taken to be the cubic B-spline function \cref{eqn:CubicSpline} and
$C(\bm{s})$ is the correction function that takes the form  
\begin{equation} \label{eqn:Correction_for_QWeight}
C(\bm{s})=\bm{\hat{H}}^T(\bm{s}) \bm{b},
\end{equation}
where $\bm{\hat{H}}(\bm{s})=[\{\bm{s}^{\bm{\alpha}} \}_{1\leq|\bm{\alpha}| \leq 2}]^T$ 
and $\bm{b}$ is a constant vector to be determined. 
For example, in two dimensions we have
\begin{equation} \notag 
\bm{\hat{H}}(\bm{s})  = [s_1, s_2, s_1^2, s_1s_2,s_2^2]^T.
\end{equation}
Substituting \cref{eqn:form_of_f,eqn:QuadratureWeight,eqn:Correction_for_QWeight}
into \cref{eqn:rep_cond_general}, we obtain the following system for $\bm{b}$: 
\begin{equation}  \label{eqn:MomentMatrix}
\bm{\widetilde{M}} \bm{b}=\widetilde{\bm{H}}_{\nd}, 
\end{equation}
where  $\bm{\widetilde{M}}$ is the moment matrix given by
\begin{equation}  \label{eqn:DiscMomentMatrix}
\bm{\widetilde{M}} = \sum\limits_{\bm{s} \in B^{\epsilon_1}_{1}(\bm{0})} \bm{\hat{H}}(\bm{s}) \rho_{}(|\bm{s}|)\phi_{}(|\bm{s}|) \bm{\hat{H}}^T(\bm{s}) . 
\end{equation}
and $\widetilde{\bm{H}}_{\nd}$ is the constant vector consisting the right hand side of \cref{eqn:DQuadMoment} when taking $f(\bm{s}) = \rho(\bm{s})\bm{\hat{H}}(\bm{s}) $.
For example, for $\nd = 2$
\begin{equation} \notag 
\widetilde{\bm{H}}_{2} = [0, 0, 1, 0, 1]^T.
\end{equation}
We note that each entry of the moment matrix is given by 
\begin{equation} \label{eqn:DMoment}
\widetilde{m}_{\alpha_1 \ldots \alpha_d}  = \sum_{\bm{s} \in B^{\epsilon_1}_{1}(\bm{0})} \bm{s}^{\bm{\alpha}}\rho_{}(|\bm{s}|) \phi_{}(|\bm{s}|)\, \text{ for } \bm{\alpha}=(\alpha_1,\ldots, \alpha_d) \text{ and } 2\leq |\bm{\alpha}| \leq 4  \,.
\end{equation}
For example, in two dimensions,  \cref{eqn:MomentMatrix} has the explicit form  
\begin{equation} \label{eqn:2dMomentMatrix}
\begin{aligned}
              \begin{bmatrix} %
               \widetilde{m}_{20} & \widetilde{m}_{11} & \widetilde{m}_{30} & \widetilde{m}_{21} & \widetilde{m}_{12}\\
               \widetilde{m}_{11} & \widetilde{m}_{02} & \widetilde{m}_{21} & \widetilde{m}_{12} & \widetilde{m}_{03}\\
                \widetilde{m}_{30} & \widetilde{m}_{21} & \widetilde{m}_{40} & \widetilde{m}_{31} & \widetilde{m}_{22}\\
                \widetilde{m}_{21} & \widetilde{m}_{12} & \widetilde{m}_{31} & \widetilde{m}_{22} & \widetilde{m}_{13}\\
               \widetilde{m}_{12} & \widetilde{m}_{03} & \widetilde{m}_{22} & \widetilde{m}_{13} & \widetilde{m}_{04}\\
              \end{bmatrix}
              \begin{bmatrix}
                 b_{10} \\
                 b_{01} \\
                 b_{20} \\
                 b_{11} \\
                 b_{02} \\
              \end{bmatrix} 
               &=
              \begin{bmatrix}
              0 \\
              0 \\
              1 \\
              0 \\
              1 \\
              \end{bmatrix},
\end{aligned}
\end{equation}
where $\bm{b} =[b_{10}, b_{01}, b_{20}, b_{11}, b_{02}]^T$.
If $\bm{\widetilde{M}}$ is invertible, then we have 
\begin{equation} \label{eqn:BadQuadratureWeights}
\omega(\bm{s}) = \phi(|\bm{s}|)\bm{\hat{H}}^T(\bm{s}) \bm{\widetilde{M}}^{-1}\widetilde{\bm{H}}_{\nd}.
\end{equation}
Otherwise, the inversion must be interpreted in a reasonable way. 
From \cref{eqn:BadQuadratureWeights}, it is unknown if the weights are strictly non-negative.
Then, under symmetry assumptions of quadrature points, we have a simple procedure to find a set of positive weights without the inversion of the moment matrix. 
It turns out the positivity  of weights $\omega(\bm{s}) $ is critical for the stability of the numerical method. 

By the symmetry assumption of the quadratic points presented in \cref{subsec:QDND}, we can show that the moment $\widetilde{m}_{\alpha_1 \ldots \alpha_{\nd}}=0$ 
if $\alpha_j$ ($j\in \{ 1,\ldots,\nd\}$) is an odd number, and $\widetilde{m}_{\alpha_1 \ldots \alpha_{\nd}}$ equals $\widetilde{m}_{\beta_1 \ldots \beta_{\nd}}$
if $(\beta_1, \ldots, \beta_{\nd})$ is a reordering of $(\alpha_1, \ldots, \alpha_{\nd})$. 
Therefore, we see immediately from \cref{eqn:2dMomentMatrix} that
\begin{equation} \notag
b_{10} = b_{01} = b_{11} = 0,
\end{equation}
and the system \cref{eqn:2dMomentMatrix} is reduced to
\begin{equation} \label{eqn:2dReducedMomentMatrix}
\begin{aligned}
  \widetilde{m}_{40} b_{20} + \widetilde{m}_{22} b_{02} &= 1, \\
 \widetilde{m}_{22} b_{20} + \widetilde{m}_{04} b_{02} &= 1. \\
\end{aligned}
\end{equation}
Notice that $\widetilde{m}_{40}=\widetilde{m}_{04}$ by the symmetry assumption and \cref{eqn:2dReducedMomentMatrix} may have multiple solutions if 
$\widetilde{m}_{40}= \widetilde{m}_{22}$. 
We can at least find one solution by adding another constraint $ b_{20}= b_{02}$. 
  Then we obtain
\begin{equation} \label{eqn:CorrectionVector}
b_{20} = b_{02} = \frac{1}{\widetilde{m}_{22}+\widetilde{m}_{40}}.
\end{equation}
Substituting \cref{eqn:CorrectionVector,eqn:Correction} into \cref{eqn:QuadratureWeight}, we arrive at an explicit expression of the quadrature weights
\begin{equation} \label{eqn:weight}
\omega_{}(\bm{s}) =\frac{\phi_{}(|\bm{s}|)}{\widetilde{m}_{22}+\widetilde{m}_{40}}|\bm{s}|^2 .
\end{equation}
In general, the weights are written as
\begin{equation}  \notag 
\omega(\bm{s}) = 
\begin{cases}
\displaystyle \frac{\phi_{}(|\bm{s}|)}{\widetilde{m}_{4}}|\bm{s}|^2 , \quad & \nd = 1, \\[10pt]
\displaystyle  \frac{\phi_{}(|\bm{s}|)}{(\nd-1)\widetilde{m}_{220\ldots0}+\widetilde{m}_{40\ldots0}}|\bm{s}|^2 , & \nd \geq 2.
\end{cases}
\end{equation}
By inspection of \cref{eqn:weight}, it is also easy to see that $\omega(\bm{s})$ only depends on $|\bm{s}|$.  Therefore we have 
\[
\omega(\bm{s}) = \omega(|\bm{s}|).
\]

\subsection{Quadrature weights using GMLS}
The way to construct weights using RK approximation as shown in \cref{subsec:QW_RK}
can be seen as a special case of a GMLS quadrature discussed in \cite{Trask2019}. This type of RKPM/MLS duality exists in the literature in many forms; while the classical RKPM and MLS shape functions are well-known to be equivalent under certain conditions \cite{chen2017meshfree}, more recent techniques such as the implicit gradient RKPM and GMLS approximation of derivatives are similarly identical \cite{hillman2016nodally,mirzaei2012generalized}. As discussed in \cite{Pasetto2019}, we show that a similar parallel holds for RKPM and GMLS nonlocal quadrature rules for completeness.

We state the GMLS problem as follows. Given a collection of points $ B^{\epsilon_1}_{1}(\bm{0}) $, we define the collection of quadrature weights $\bm{\omega}^T = \left\{\omega(\bm{s})\right\}_{\bm{s} \in B^{\epsilon_1}_{1}(\bm{0})}$ via the equality constrained optimization problem:

\begin{equation} \label{eqn:GMLS}
\frac{1}{2} \underset{\omega \in \mathbb{R}^{N_{\textnormal{d}}}}{argmin} \sum\limits_{\bm{s} \in B^{\epsilon_1}_{1}(\bm{0})}   \omega^2(\bm{s}) \frac{1}{W(|\bm{s}|)}\\
\end{equation}
\begin{align*}
\text{s.t.}
\sum\limits_{\bm{s} \in B^{\epsilon_1}_{1}(\bm{0})} f(\bm{s}) \omega(\bm{s}) = \int_{B_{1}(\bm{0})} f(\bm{s}) d\bm{s}, \, \forall f \in \mathbf{V},
\end{align*}
where $\mathbf{V}$ denotes a Banach space of integrands to be integrated exactly, and $W(r)$ is a radially symmetric positive weight function supported on $B_1(\bm{0})$. 
Here we select $\mathbf{V}$ as the space of functions in the form of \cref{eqn:form_of_f}. The solution to \cref{eqn:GMLS} is then given explicitly by the saddle point problem

\begin{equation} \label{eqn:gmlsKKT}
\begin{bmatrix}
\bm{W}^{-1} & (\bm{\hat{H}} \bm{\rho})^T \\
\bm{\hat{H}} \bm{\rho} & \bm{0}
\end{bmatrix}
\begin{bmatrix}
\bm{\omega} \\
\bm{\lambda}
\end{bmatrix}
 = 
 \begin{bmatrix}
 0 \\
 \widetilde{\bm{H}}_{\nd}
 \end{bmatrix}\; ,
\end{equation}
where $\bm{\lambda} \in \mathbb{R}^{\textnormal{dim}(\mathbf{V})}$ denotes a vector of Lagrange multipliers used to enforce the constraint, $\bm{W}$ denotes a $N_{\nd} \times N_{\nd} $ diagonal matrix with diagonal entries $\left\{ W(|\bm{s}|) \right \} _{\bm{s}\in B^{\epsilon_1}_{1}(\bm{0})}$, $\bm{\hat{H}}$ denotes a $\textnormal{dim}(\bm{V}) \times N_{\nd} $ rectangular matrix with column vectors $\left\{ \bm{\hat{H}}(\bm{s}) \right \} _{\bm{s}\in B^{\epsilon_1}_{1}(\bm{0})}$, and $\bm{\rho}$ denotes a $N_{\nd} \times N_{\nd} $ diagonal matrix with diagonal entries $\left\{\rho(|\bm{s}|)\right\}_{\bm{s} \in B^{\epsilon_1}_{1}(\bm{0})}$. Solution of this system yields the following expression for the quadrature weights
\begin{equation}
\begin{aligned}
\bm{\omega} &= \bm{W} (\bm{\hat{H}} \bm{\rho})^T \left( (\bm{\hat{H}} \bm{\rho})  \bm{W} (\bm{\hat{H}} \bm{\rho})^T \right)^{-1} \widetilde{\bm{H}}_{\nd}, \\
&= \bm{W}\bm{\rho} \bm{\hat{H}}^T  \left( \bm{\hat{H}} \bm{\rho} \bm{W}\bm{\rho} \bm{\hat{H}}^T \right)^{-1} \widetilde{\bm{H}}_{\nd}.
\end{aligned}
\end{equation}
If we let $W(|\bm{s}|)\rho(|\bm{s}|) = \phi(|\bm{s}|)$,  a direct comparison to \cref{eqn:BadQuadratureWeights} reveal that the two are algebraically equivalent, depending upon how the matrix inverse is handled. In \cite{Trask2019}, the authors used a pseudoinverse to handle the lack of uniqueness in the resulting solution.




There are several consequences for this equivalence. First, it reveals that the lack of invertibility of the moment matrix $\bm{\widetilde{M}}$ in \cref{eqn:DiscMomentMatrix} may be interpreted as a nonunique solution to \cref{eqn:GMLS}, meaning that there are multiple choices of quadrature weights providing the desired reproduction properties. From the construction in the previous section, we know that at least one of those solutions corresponds to positive quadrature weights. We may thus add an inequality constraint to \cref{eqn:GMLS} to enforce positivity, due to the existence of a non-empty feasible set. This is in contrast to existing literature \cite{Trask2019}, whereby no guarantees were made regarding positivity of quadrature weights. Of course, this result holds only for uniform grids, and future work may focus on whether such results hold for general quasi-uniform particle distributions in which \cite{Trask2019} is applied.

In light of this GMLS/RK equivalence, the stability analysis in subsequent sections will apply equally to these previous works, and existing error analysis in the literature related to GMLS approaches likewise may be applied to the current scheme, under appropriate assumptions. Thus, the substantial literature pursuing both RK and MLS as platforms for establishing asymptotic compatibility are effectively equivalent.

\subsection{Truncation error of the quasi-discrete nonlocal operator}
We have constructed a quasi-discrete nonlocal diffusion operator using meshfree integration and we next study the associated truncation error. 

\begin{lemma} \label{lem:DiffernceNonlocal}
Assume $u \in C^4(\R^{\nd})$, then for any $\bm{x}\in\R^\nd$
\beq
\left| \mathcal{L}^{\epsilon}_{\delta} u(\bm{x})- \mathcal{L}_{\delta} u(\bm{x}) \right| \leq C\delta^2 | u^{(4)} |_{\infty}.
\eeq
\end{lemma}
\begin{proof}   
Using Taylor's theorem, we have 
\begin{equation} \label{eqn:Remainder}
u(\bm{x}+\bm{s}) + u(\bm{x}-\bm{s}) - 2u(\bm{x})  = 2\sum_{|\bm{\alpha}|=2}\bm{s}^{\bm{\alpha}} \frac{D^{\bm{\alpha}}u(\bm{x})}{\bm{\alpha}!} + \sum_{|\bm{\beta|=4}} \bm{s}^{\bm{\beta}} \frac{R_{\bm{\beta}}(\bm{y})}{\bm{\beta}!}\,,
\end{equation}
where $|R_{\bm{\beta}}(\bm{y})|\leq C |u^{(4)}|_\infty$ and $\bm{y}=\bm{y}(\bm{x},\bm{s})$ that depends on both $\bm{x}$ and $\bm{s}$.
Therefore, for any point $\bm{x} \in  \R^\nd$, 
\begin{equation} \label{eqn:Ldelep-Ldel}
\begin{aligned}
\left| \mathcal{L}^{\epsilon}_{\delta} u(\bm{x}) - \mathcal{L}_{\delta} u(\bm{x}) \right|  &=\Bigg\lvert\sum_{|\bm{\alpha}|=2}\frac{D^{\bm{\alpha}}u(\bm{x})}{\bm{\alpha}!} \left( \sum_{\bm{s} \in B^{\epsilon}_{\delta}(\bm{0})} \omega_{\delta}(\bm{s}) \rho_{\delta}(|\bm{s}|)\bm{s}^{\bm{\alpha}}  - \int_{ B_{\delta}(\bm{0})} \rho_{\delta}(|\bm{s}|) \bm{s}^{\bm{\alpha}}  d\bm{s}\right)  \\
  + \sum_{|\bm{\beta}|=4}&\frac{1}{2\bm{\beta}!} \left( \sum_{\bm{s} \in B^{\epsilon}_{\delta}(\bm{0})} \omega_{\delta}(\bm{s}) \rho_{\delta}(|\bm{s}|)\bm{s}^{\bm{\beta}}R_{\bm{\beta}}(\bm{y})- \int_{ B_{\delta}(\bm{0})} \rho_{\delta}(|\bm{s}|) \bm{s}^{\bm{\beta}}  R_{\bm{\beta}}(\bm{y})d\bm{s}\right) \Bigg\rvert \\
  \leq 0 + C&| u^{(4)} |_{\infty}  \sum_{|\bm{\beta}|=4} \frac{1}{2\bm{\beta}!}\left( \sum_{\bm{s} \in B^{\epsilon}_{\delta}(\bm{0})} \omega_{\delta}(\bm{s}) \rho_{\delta}(|\bm{s}|)|\bm{s}|^{4} +  \int_{ B_{\delta}(\bm{0})} \rho_{\delta}(|\bm{s}|) |\bm{s}|^{4} d\bm{s}\right)  \\
  \leq C& \delta^2 | u^{(4)} |_{\infty}.
\end{aligned}
\end{equation}
where we have used \cref{eqn:DQuadMoment}. 
\end{proof}

\Cref{lem:DiffernceNonlocal} suggested that the upper bound of the truncation error between $\mathcal{L}_\delta^\epsilon u$ and $\mathcal{L}_\delta u$ is fixed if 
the number of quadrature points inside the $\delta$-neighborhood of each nodal point does not change, and that the truncation error goes to zero in second order as $\delta$ goes to zero. In the following remark, we also provide, formally, another observation of the truncation error corresponds to $\epsilon$. 

\begin{remark}
With additional regularity assumptions on $u$ and $\rho_\del$, it is possible to show that 
\beq \label{eqn:RemarkOnLdelep}
\left| \mathcal{L}^{\epsilon}_{\delta} u(\bm{x})- \mathcal{L}_{\delta} u(\bm{x}) \right| \leq C\ep^2\,,
\eeq
where $\ep$ is viewed as the spacing the qudrature points as depicted in \cref{fig:DiscreteBall}. Indeed, if we assume $F_\del(\bm{s}):= \rho_{\delta}(|\bm{s}|)\bm{s}^{\bm{\beta}}R_{\bm{\beta}}(\bm{y}(\bm{x},\bm{s}))\in C^2(\overline{B_\del(\bm{0})})$ for $|\bm{\beta}|=4$ (this can be achieved by assuming {\it e.g.}, $u\in C^6(\mathbb{R}^\nd)$ and $\rho_{\delta}(|\bm{s}|)\bm{s}^{\bm{\beta}}\in C^2(\overline{B_\del(\bm{0})})$), then $\sum_{\bm{s} \in B^{\epsilon}_{\delta}(\bm{0})} \omega_{\delta}(\bm{s})F_\del(\bm{s}) $ is an $O(\ep^2)$ approximation to $\int_{B_\del(\bm{0})} F_\del(\bm{s}) d\bm{s}$, which can be used to estimate the second line in \cref{eqn:Ldelep-Ldel}.
This is to say that for a fixed horizon $\delta$, the  quasi-discrete operator $\cL_\del^\ep$ converges to   the nonlocal diffusion operator $\mathcal{L}_{\delta}$ only if $\epsilon\to0$, where the number of quadrature points used inside the $\del$-neighborhood of each nodal point should approach infinity for it to happen. 
\end{remark}

\cref{lem:DiffernceNonlocal} shows that if the quasi-discrete 
operator $\mathcal{L}^{\epsilon}_{\delta}$ satisfy the polynomial reproducing conditions \cref{eqn:DQuadMoment} up to second order, then $\mathcal{L}^{\epsilon}_{\delta}$
is a second-order approximation of $\mathcal{L}_{\delta}$ in $\delta$.
For high order approximations, one could follow the same procedure 
to design weights such that the quasi-discrete operator satisfy high order polynomial reproducing conditions. However,  the positivity of weights for high order approximations  needs 
further investigation which is beyond the scope of this paper. 
We will see in the next section that the positivity of weights is crucial to 
guarantee the stability of numerical schemes applied to the quasi-discrete operator.

\section{Convergence analysis of the RK collocation on the quasi-discrete nonlocal diffusion} \label{sec:DCAnalysis}

In this section, we apply the RK collocation method introduced in \cref{sec:RKCollocation} to the quasi-discrete nonlocal diffusion operator defined in \cref{sec:QuasiDiscrete}. 
The RK collocation scheme for the quasi-discrete operator is formulated as follows. 
Find a function $u \in S \left(\square \cap {\Omega}  \right) $ such that
\begin{equation} \label{eqn:DCollocationScheme_1}
-\mathcal{L}^\epsilon_\delta u(\bm{x_k}) =  {f_\del}(\bm{x_k}), \, \quad \bm{x}_{\bm{k}} \in (\square \cap \Omega)\,.
\end{equation}
Equivalently, \cref{eqn:DCollocationScheme_1} is also written as
\begin{equation}  \label{eqn:DCollocationScheme}
-r^h_{\Omega} \mathcal{L}^{\epsilon}_{\delta} u = r^h_{\Omega} {f_\del}, \quad  u \in S(\square \cap {\Omega}). 
\end{equation}
For practical reasons, in this section we assume $\delta = M_0 h_{\max}$, where $M_0>0$ is fixed and investigate the convergence behaviour of the numerical solution of the collocation scheme, obtaining results similar to \cref{sec:CAnalysis}.

\subsection{Stability of RK collocation on the quasi-discrete nonlocal diffusion}
In this subsection, we aim to show the the stability of the RK collocation scheme \cref{eqn:DCollocationScheme} given as follows.

\begin{theorem} \label{thm:DStability}
\textbf{(Stability II)} For any $\delta \in(0, \delta_0]$ and $u \in S(\square \cap {\Omega})$, we have
\begin{equation} \notag
\left|r^h_{\Omega}(-\mathcal{L}^{\epsilon}_{\delta} u)\right|_{h} \geq C \| u\| _{L^2(\mathbb{R}^{\nd})}\,,
\end{equation}
where $C$ is a constant that only depends on $\Omega$, $\delta_0$, and $M_0$.
\end{theorem}

The key for showing the theorem is to establish an analogue of \Cref{lem:GCF}. Similarly as \cref{eqn:FT_diffusion}, we can find Fourier symbol of the  operator $\mathcal{L}^{\epsilon}_{\delta}$,
\begin{equation} \notag
\begin{aligned}
-\widehat{\mathcal{L}^{\epsilon}_{\delta} u}(\bm{\xi}) 
  &= \lambda^{\epsilon}_{\delta}(\bm{\xi}) \widehat{u}(\bm{\xi}),
\end{aligned}
\end{equation}
where $\lambda^{\epsilon}_{\delta}(\bm{\xi})$ is given by
\begin{equation} \label{eqn:DNonlocalEigen} 
\lambda^{\epsilon}_{\delta}(\bm{\xi}) =\sum\limits_{\bm{s} \in B^{\epsilon}_{\delta}(\bm{0})} \omega_{\delta}(\bm{s}) \rho_{\delta}(\bm{s}) (1-\textnormal{cos}(\bm{s}\cdot\bm{\xi})).
\end{equation}
Since $\omega_{\delta}({\bm{s}})$ is symmetric and  non-negative, $\lambda^{\epsilon}_{\delta}(\bm{\xi}) $ is real and non-negative. We then obtain the Fourier representation of the collocation scheme as follows.

\begin{lemma}  \label{lem:FDC}
Let $\widetilde{u}(\bm{\xi})$ and $\widetilde{v}(\bm{\xi})$ be the Fourier series of the sequences $(u_{\bm{k}}), (v_{\bm{k}}) \in l^2(\mathbb{Z}^d)$ respectively. Then
\[
((u_{\bm{k}}), -r^h\mathcal{L}^{\epsilon}_{\delta}i^h(v_{\bm{k}}))_{l^2} = (2\pi)^{-{\nd}} \int_{\bm{Q}} \widetilde{u}(\bm{\xi}) \overline{\widetilde{v}(\bm{\xi})} \lambda^{\epsilon}_C (\delta, \bm{h}, \bm{\xi}) d\bm{\xi},
\] \label{FDCollocation}
where $\lambda^{\epsilon}_C$ is defined as
\begin{equation} \label{eqn:lambdaDC}
\lambda^{\epsilon}_C(\delta, \bm{h}, \bm{\xi}) = 2^{4{\nd}} \sum_{\bm{r} \in \mathbb{Z}^{\nd}} \lambda^{\epsilon}_{\delta}\left((\bm{\xi} + 2 \pi \bm{r}) \oslash{\bm{h}} \right) \prod_{j=1}^{\nd} h_j \left(\frac{\textnormal{sin}(\xi_j /2)}{\xi_j + 2 \pi r_j} \right)^{4}. \\
\end{equation}
Moreover,
\begin{equation} \label{eqn:CC}
\lambda_{C} (\delta, \bm{h}, \bm{\xi}) \leq C \lambda^{\epsilon}_{C} (\delta, \bm{h}, \bm{\xi}),
\end{equation}
for some generic constant $C >0$.
\end{lemma}
\begin{proof} The derivation of \cref{eqn:lambdaDC} is similar to \cref{eqn:lambdaC}, following the replacement of $\lambda_{\delta}(\bm{\xi}+ 2 \pi \bm{r})$ by $\lambda^{\epsilon}_{\delta}(\bm{\xi}+ 2 \pi \bm{r})$. We proceed to show \cref{eqn:CC}. By change of variables, we obtain
\begin{equation} \notag  
\lambda_{\delta} \left((\bm{\xi}+2\pi\bm{r})\oslash{\bm{h}} \right) =  \frac{1}{\delta^2}\lambda_1\left(\delta (\bm{\xi}+2\pi\bm{r})\oslash{\bm{h}} \right)
\end{equation}
and 
\begin{equation} \notag  
\lambda^{\epsilon}_{\delta} \left((\bm{\xi}+2\pi\bm{r})\oslash{\bm{h}} \right) =  \frac{1}{\delta^2}\lambda^{\epsilon_1}_1\left(\delta (\bm{\xi}+2\pi\bm{r})\oslash{\bm{h}} \right),
\end{equation}
where 
\begin{equation} \notag  
\lambda_{1} \left(\delta(\bm{\xi}+2\pi\bm{r})\oslash{\bm{h}} \right) = \mathlarger{\int}_{B_{1}(\bm{0})} \rho (|\bm{s}|) \left(1-\textnormal{cos}\left( \delta \bm{s}  \cdot \left((\bm{\xi}+2\pi\bm{r}) \oslash \bm{h} \right) \right) \right) d\bm{s} 
\end{equation}
and 
\begin{equation} \notag 
\lambda^{\epsilon_1}_{1}(\delta(\bm{\xi}+2\pi\bm{r}) \oslash{\bm{h}}) = \sum\limits_{\bm{s} \in B^{\epsilon_1}_{1}(\bm{0})} \omega(|\bm{s}|) \rho(|\bm{s}|) \left(1-\textnormal{cos}\left( \delta \bm{s}  \cdot \left((\bm{\xi}+2\pi\bm{r}) \oslash \bm{h} \right) \right) \right).
\end{equation}
Let us decompose the set $\bm{Q}=(-\pi,\pi)^\nd$ into $\bm{Q}_1$ and  $\bm{Q}_2$,
\begin{equation} \notag  
\bm{Q}_1 : = \{ \bm{\xi} \in \bm{Q} : \frac{ \delta |\bm{\xi}|}{h_{\min}} \leq \pi \}, \text{ and } \bm{Q}_2= \overline{\bm{Q} \backslash \bm{Q}_1}. 
\end{equation}
Notice that for  $\bm{\xi} \in \bm{Q}_1$ and $\bm{s}\in B_1(\bm{0})$,
\[
 \big|  \delta \bm{s}  \cdot \left(\bm{\xi}  \oslash \bm{h} \right)  \big| \leq  \frac{\delta |\bm{\xi}|}{h_{\textnormal{min}}}  \leq \pi\,.
\]

First, by \cref{eqn:lambdaDC}, we observe that 
\[
\lambda^{\epsilon}_{C} (\delta, \bm{h}, \bm{\xi}) \geq C  \frac{1}{\delta^2}\lambda^{\epsilon_1}_1\left(\delta\bm{\xi}\oslash{\bm{h}} \right) \prod_{j=1}^{\nd} h_j \left(\frac{\textnormal{sin}(\xi_j /2)}{\xi_j } \right)^{4} \,.
\]
Notice that there exists $C>0 $ such that for $x \in (-\pi, \pi)$
\[
1-\textnormal{cos}(x) \geq C x^2\,.
\]
Then, we have for $\bm{\xi} \in \bm{Q}_1$,
\begin{equation} \label{eqn:DCsymbol}
\begin{aligned}
\lambda^{\epsilon_1}_{1} \left(\delta\bm{\xi} \oslash \bm{h} \right) & = \sum\limits_{\bm{s} \in B^{\epsilon_1}_{1}(\bm{0})} \omega(|\bm{s}|) \rho(|\bm{s}|) \left(1-\textnormal{cos}\left( \delta \bm{s} \cdot \left(\bm{\xi} \oslash \bm{h} \right) \right) \right),\\  
& \geq C \del^2\sum\limits_{\bm{s} \in B^{\epsilon_1}_{1}(\bm{0})} \omega(|\bm{s}|) \rho(|\bm{s}|) | \bm{s} \cdot \bm{\xi} \oslash \bm{h}|^2  \\
& \geq C \left(\frac{\delta |\bm{\xi}|}{h_{\textnormal{max}}} \right)^2\sum\limits_{\bm{s} \in B^{\epsilon_1}_{1}(\bm{0})} \omega(|\bm{s}|) \rho(|\bm{s}|) |\bm{s}|^2  \textnormal{cos}^2(\theta(\bm{s}, \bm{\xi} \oslash \bm{\hat{h}})) ,  \\
& \geq C |\bm{\xi}|^2,
\end{aligned}
\end{equation}
where $C$ depends only on $M_0=\delta/h_{\textnormal{max}}$ and the set $B_1^{\epsilon_1}(\bm{0})$ and $\theta=\theta(\bm{s},\bm{\xi} \oslash \bm{\hat{h}})$ is the angle between $\bm{s} $ and $\bm{\xi}\oslash {\bm{\hat{h}}}$. The last line of of \cref{eqn:DCsymbol} comes from the following observation. For a fixed vector $\bm{\xi} \oslash {\bm{\hat{h}}} $, $\textnormal{cos}(\theta (\bm{s}, \bm{\xi} \oslash \bm{\hat{h}})) = 0 $ only for the points $\bm{s}$ that lies in directions orthogonal to $\bm{\xi}\oslash {\bm{\hat{h}}}$. But from the symmetry assumption of the discrete set $B^{\epsilon_1}_1(\bm{0})$ in \cref{subsec:QW_RK}, there are always $\bm{s}\in B^{\epsilon_1}_1(\bm{0}) $ such that $\bm{s}$ 
is not orthogonal to $\bm{\xi} \oslash {\bm{\hat{h}}} $. 
Therefore, for any nonzero $\bm{\xi} \in \bm{Q}_1$, the summation in the second line of 
 \cref{eqn:DCsymbol} is always a positive number. 
 Then it has a positive lower bound since  $\bm{Q}_1$ is a compact set. 
 Now for $\bm{\xi}\in \bm{Q}_2$, we have $|\bm{\xi} \oslash \bm{\hat{h}}| \geq |\bm{\xi}| \geq \pi h_{\min}/(h_{\max} M_0) $, then
\begin{equation}
\begin{aligned}
\lambda^{\epsilon_1}_{1} \left(\delta\bm{\xi} \oslash \bm{h} \right) & = \sum\limits_{\bm{s} \in B^{\epsilon_1}_{1}(\bm{0})} \omega(|\bm{s}|) \rho(|\bm{s}|) \left(1-\textnormal{cos}\left( \frac{\delta}{h_{\max}} \bm{s}  \cdot \left(\bm{\xi} \oslash \bm{\hat{h}} \right) \right) \right),\\ 
& =  \sum\limits_{\bm{s} \in B^{\epsilon_1}_{1}(\bm{0})} \omega(|\bm{s}|) \rho(|\bm{s}|) \left(1-\textnormal{cos}\left( M_0 |\bm{s}| | \bm{\xi} \oslash \bm{\hat{h}} | \cos({\theta(\bm{s},\bm{\xi} \oslash \bm{\hat{h}} ))} \right) \right) \,.\\
\end{aligned}
\end{equation}
As we have seen, for any fixed $\bm{\xi} \oslash \bm{\hat{h}} $,  there always exists 
$\bm{s}\in B^{\epsilon_1}_{1}(\bm{0}) $
such that $\cos({\theta(\bm{s},\bm{\xi} \oslash \bm{\hat{h}} ))} \neq 0$. 
However,  $\lambda^{\epsilon_1}_{1} \left(\delta\bm{\xi} \oslash \bm{h} \right)$ 
may still be zero if $\bm{s}\in B^{\epsilon_1}_{1}(\bm{0})$ and 
$M_0 |\bm{s}| | \bm{\xi} \oslash \bm{\hat{h}} | \cos({\theta(\bm{s},\bm{\xi} \oslash \bm{\hat{h}} ))} =2k\pi $ for $k\in \mathbb{Z}$. If this happens, one can add another point $\tilde{\bm{s}} \in B^{\epsilon_1}_{1}(\bm{0})$ in the same direction of $\bm{s}$ such that 
$|\bm{s}|/|\tilde{\bm{s}}|$ is an irrational number and thus
$M_0 |\tilde{\bm{s}}| | \bm{\xi} \oslash \bm{\hat{h}} | \cos({\theta(\tilde{\bm{s}},\bm{\xi} \oslash \bm{\hat{h}} ))} \neq 2k\pi $ for any $k\in \mathbb{Z}$.
Therefore, for a proper choice of $B^{\epsilon_1}_{1}(\bm{0}) $ with $N_\nd \geq 4\nd$,
we can always have $\lambda^{\epsilon_1}_{1} \left(\delta\bm{\xi} \oslash \bm{h} \right)> 0$ 
for $\bm{\xi}\in \bm{Q}_2$. Then since $\bm{Q}_2$ is a compact set, we have $\lambda^{\epsilon_1}_{1} \left(\delta\bm{\xi} \oslash \bm{h} \right) \geq C \geq |\bm{\xi}|^2$ for $\bm{\xi}\in \bm{Q}_2$.
 Now observe that
for any nonzero $\bm{\xi}=(\xi_1,\ldots,\xi_\nd) \in \bm{Q}$, we have
\[
C_1 < \left(\frac{\textnormal{sin}({\xi}_j/2)}{{\xi}_j} \right)^{4} < C_2, \; j\in \{1,\ldots,\nd\}\,,
\]
\noindent 
where $C_1, C_2 > 0$ are generic constants. Then we arrive at  
\begin{equation} \label{eqn:CDbig} 
\begin{aligned}
\lambda^{\epsilon}_{C} (\delta, \bm{h}, \bm{\xi})  \geq C\left( \frac{|\bm{\xi}|}{\delta}\right)^2  \prod_{j=1}^{\nd} h_j\,,
\end{aligned}
\end{equation}
for $\bm{\xi}\in\bm{Q}$. 

Next, use the fact that 
\[
1-\textnormal{cos}(x) \leq x^2, \quad \textnormal{for } x\geq 0,
\]
to obtain, for any $\bm{r} \in \mathbb{Z}^{\textnormal{d}}$,
\begin{equation} \notag 
\begin{aligned}
\lambda_{1} \left(\delta(\bm{\xi}+2\pi\bm{r}) \oslash \bm{h} \right) & \leq \int_{B_{1}(\bm{0})} \rho(|\bm{s}|) \delta^2 |\bm{s}|^2 |(\bm{\xi}+2\pi\bm{r}) \oslash {\bm{h}} |^2 d\bm{s}, \\
& \leq C \left(\frac{\delta |(\bm{\xi}+2\pi\bm{r})|}{h_{\textnormal{min}}} \right)^2\int_{B_{1}(\bm{0})} \omega(|\bm{s}|) \rho(|\bm{s}|) |\bm{s}|^2  d\bm{s} ,  \\
& \leq C  |(\bm{\xi}+2\pi\bm{r})|^2\,,
\end{aligned}
\end{equation}
where we have used $M_0=\delta/h_{\max}$ is fixed and that $\bm{h}$ is quasi-uniform ($h_{\max}/h_{\min}$ is bounded above and below). 
Hence we have 
\begin{equation} \notag 
\begin{aligned}
&\sum_{\bm{r} \in \mathbb{Z}^{\nd}} \lambda_1\left(\delta(\bm{\xi} + 2 \pi \bm{r}) \oslash{\bm{h}} \right) \prod_{j=1}^{\nd} h_j \left(\frac{\textnormal{sin}(\xi_j /2)}{\xi_j + 2 \pi r_j} \right)^{4} \\
& \leq C \sum_{\bm{r}\in \mathbb{Z}^{\nd}} |\bm{\xi}+2\pi\bm{r}|^{2}\prod_{j=1}^{\nd} h_j \left(\frac{\textnormal{sin}(\xi_j /2)}{\xi_j + 2 \pi r_j} \right)^{4},  \\
& \leq C|\bm{\xi}|^4 \sum_{\bm{r}\in \mathbb{Z}^{\nd}} |\bm{\xi}+2\pi\bm{r}|^{2}\prod_{j=1}^{\nd} h_j \left(\frac{1}{\xi_j + 2 \pi r_j} \right)^{4}\\
& \leq  C |\bm{\xi}|^4 \prod_{j=1}^{\nd} h_j  \sum_{\bm{r} \in \mathbb{Z}^{\nd}} \prod_{j=1}^{\nd} \frac{1}{ |{\xi}_j + 2 \pi {r_j}|^2}, \\
& \leq  C |\bm{\xi}|^2\prod_{j=1}^{\nd} h_j .
\end{aligned}
\end{equation}
Immediately, we have the following bound for $\lambda_C(\delta, \bm{h}, \bm{\xi})$ 
\begin{equation} \label{eqn:CCsmall}
\lambda_{C} (\delta, \bm{h}, \bm{\xi})  \leq C  \left( \frac{|\bm{\xi}|}{\delta} \right)^2   \prod_{j=1}^{\nd} h_j. 
\end{equation}
Finally, \cref{eqn:CC} is shown by combining \cref{eqn:CDbig} and \cref{eqn:CCsmall}. \end{proof}

\begin{proof}[Proof of \cref{thm:DStability}] 
By applying \cref{lem:FDC}, the proof follows similarly to the proof of \cref{thm:stability}.
\end{proof}

\subsection{Convergence of the RK collocation for quasi-discrete nonlocal diffusion} 
In this subsection, we establish the convergence of the RK scheme \cref{eqn:DCollocationScheme} to the 
corresponding local problem as 
$h_{\max}\to 0 $ with a fixed ratio $M_0 =\delta / h_{\max}$. 
We show first the discrete model error between the quasi-discrete nonlocal diffusion model and its local limit.  
\begin{lemma} \label{lem:DConsistency}
\textbf{(Discrete model error II)} Assume $u(\bm{x}) \in C^4(\R^{\nd})$, then
\begin{equation} \notag 
\left| r^h\mathcal{L}^{\epsilon}_{\delta}\Pi^hu - r^h \mathcal{L}_{0} u \right|_{h} \leq C  |u^{(4)}|_{\infty} (h_{\max}^2 + \delta^2).
\end{equation}
\end{lemma}
\begin{proof} In order to prove this Lemma, we need an intermediate result. Similar to the proof of \cref{lem:consistency}, for $\bm{x_k} \in \square $,   
\begin{equation} \label{eqn:DiscreteOnRK}
\begin{aligned}
\left| \mathcal{L}^{\epsilon}_{\delta} \Pi^h u(\bm{x_k}) - \mathcal{L}^{\epsilon}_{\delta} u(\bm{x_k}) \right| 
&=  \left| \sum_{\bm{s} \in B^{\epsilon}_{\delta}(\bm{0})} \omega_{\delta}(\bm{s}) \rho_{\delta}(\bm{s}) \left( E(\bm{x_k}+\bm{s}) - E(\bm{x_k}) \right) \right|, \\
 & \leq C  h_{\textnormal{max}}^2 |u^{(4)}|_{\infty} \sum_{\bm{s} \in B^{\epsilon}_{\delta}(\bm{0})} \omega_{\delta}(\bm{s}) \rho_{\delta}(\bm{s}) |\bm{s}|^2 , \\
 & \leq C h_{\textnormal{max}}^2 | u^{(4)} |_{\infty} ,
\end{aligned}
\end{equation}

Then, by combining \cref{eqn:DiscreteOnRK} and  \cref{lem:DiffernceNonlocal} and using the RK interpolation, the discrete model error of collocation scheme is given as
\begin{align*} 
\left| r^h\mathcal{L}^{\epsilon}_{\delta} \Pi^hu -  r^h \mathcal{L}_{0} u \right|_{h} &\leq \left| r^h \mathcal{L}^{\epsilon}_{\delta}  \Pi^hu -  r^h \mathcal{L}^{\epsilon}_{\delta} u  \right|_{h}  + \left| r^h \mathcal{L}^{\epsilon}_{\delta}  u - r^h  \mathcal{L}_{\delta} u \right|_{h} \\  
& \quad + \left| r^h  \mathcal{L}_{\delta}  u - r^h \mathcal{L}_{0} u  \right|_{h}, \\ 
&\leq C (h^2_{\textnormal{max}}+\delta^2 + \delta^2)| u^{(4)} |_{\infty} . 
\end{align*}
\end{proof}

Combining \cref{thm:DStability} and \cref{lem:DConsistency}, the numerical solution of \cref{eqn:DCollocationScheme} converges to its local limit. We present the theorem without proof since it is similar to the proof of \cref{thm:AC}.

\begin{theorem}  \label{thm:QAC}
Assume the local exact solution $u^0$ is sufficiently smooth, i.e., $u^0 \in C^4(\overline{ \Omega_{\delta_0}})$. For any $\delta \in (0, \delta_0]$, let $u^{\delta,  \epsilon, h}$ be the numerical solution of the collocation scheme with meshfree integration \cref{eqn:DCollocationScheme} and fix the ratio between $\delta$ and $h_{\max}$. Then,
\begin{equation} \notag 
\|u^0 - u^{\delta, \epsilon, h} \|_{L^2(\Omega)} \leq C (h_{\max}^2 + \delta^2).
\end{equation}
\end{theorem}

\section{Numerical Example} \label{sec:NumericalExample}
In this section, numerical examples in two dimensions are conducted to validate the convergence analysis in the previous sections. We let $\Omega =  (0,1)^2$, and use the manufactured solution $u(x_1,x_2) = x_1^2(1-x_1^2)+x_2^2(1-x_2^2)$ to calculate $f_0 = -\cL_0 u $ and $f_\del = - \cL_\del u$, such that,
\begin{equation} \notag 
f_0(\bm{x}) = 12(x_1^2+x_2^2)-4 \textnormal{ and } f_\delta(\bm{x}) = f_0(\bm{x}) + 2\delta^2.
\end{equation}
We investigate the convergence rate of the RK collocation scheme in \cref{sec:RKCollocation} for the following nonlocal equation
\begin{equation}  \label{eqn:NLEx}
\begin{cases}
-\mathcal{L}_{\delta} u(\bm{x}) = f_\delta(\bm{x}), & \bm{x} \in \Omega, \\
u(\bm{x}) = x_1^2(1-x_1^2)+x_2^2(1-x_2^2), & \bm{x} \in \Omega_{\cI}.
\end{cases}
\end{equation}
The exact solution of \cref{eqn:NLEx} is given by the manufactured solution $u$. 
The nonlocal kernel is chosen as $\rho_{\delta}(|\bm{s}|) = \frac{4}{\pi \delta^4} \chi(|\bm{s}\leq 1|)$.
To verify the asymptotic compatibility of our scheme where $\delta$ also goes to zero, 
we replace the right hand side $f_\delta$ of \eqref{eqn:NLEx} with $f_0$,
and test the convergence of the numerical solution of the following nonlocal diffusion problem 
\begin{equation}  \label{eqn:NTLEx}
\begin{cases}
-\mathcal{L}_{\delta} u(\bm{x}) = f_0(\bm{x}), & \bm{x} \in \Omega, \\
u(\bm{x}) = x_1^2(1-x_1^2)+x_2^2(1-x_2^2), & \bm{x} \in \Omega_{\cI},
\end{cases}
\end{equation}
to the solution of the local problem given by 
\begin{equation}  \label{eqn:NTLEx_local}  
\begin{cases}
-\Delta u(\bm{x}) = f_0(\bm{x}), & \bm{x} \in \Omega, \\
 \quad \quad u(\bm{x}) = x_1^2(1-x_1^2)+x_2^2(1-x_2^2), & \bm{x} \in  \partial \Omega.
\end{cases}
\end{equation}
We apply the two collocation schemes in \cref{sec:RKCollocation,sec:QuasiDiscrete} and investigate their convergence properties in \cref{subsec:collocation_numer,subsec:meshfree_numer} respectively. 
{To implement the non-homogeneous boundary condition for the nonlocal problem, we use the RK interpolation of exact boundary values on $\Om_\cI$. Notice that the boundary conditions in \cref{eqn:NLEx,eqn:NTLEx} are imposed such that the regularity assumptions of exact solutions across the extended domain $\overline{\Om\cup\Om_\cI}$ are satisfied, as stated in the convergence theorems in \cref{sec:CAnalysis,sec:DCAnalysis},. In \cref{subsec:boundary_numer}, we use numerical examples to demonstrate that convergence rates may be compromised if the boundary conditions are not imposed appropriately to meet the regularity assumptions.

\subsection{RK collocation} \label{subsec:collocation_numer}
We first use the RK collocation method as described in \cref{sec:RKCollocation} and choose the discretization parameter as $h_1 = 2h_2$ (so $h_{\textnormal{max}} = h_1$), and then study the convergence of the numerical the collocation scheme \cref{eqn:CollocationScheme}. 
Numerical integration needs to be performed in order to obtain the stiffness matrix of the scheme  \cref{eqn:CollocationScheme}. 
We use Gauss quadrature points proposed in \cite{Pasetto2018} for the numerical integration on circular regions in the neighborhood of radius $\delta$ of each nodal point $\bm{x}_{\bm{k}}\in (\square \cap \Omega)$. To avoid integration error, we use  $\approx 10^3$ 
Gauss quadrature points in the $\delta$-neighborhood of each nodal point. More details on the discussion of Gauss quadrature points can be found in \cite{Pasetto2018}. 

Convergence profiles are shown in \cref{fig:convergence}. We investigate the convergence behaviour when the nonlocal length scale $\delta$ is coupled with discretization parameter $h_{\textnormal{max}}$ in various ways. When $\delta$ is fixed, we solve the nonlocal equation \cref{eqn:NLEx} and the numerical solution converges to the nonlocal solution  at a second-order convergence rate. When both $\delta$ and $h_{\textnormal{max}}$ are changing, we solve the nonlocal equation \cref{eqn:NTLEx} the numerical solution converges to the local limit \cref{eqn:NTLEx_local}.  When $\delta$ goes to zero faster ($\delta = h_{\textnormal{max}}^2$) and at the same rate as $h_{\textnormal{max}}$, ($\delta = h_{\textnormal{max}}$), second-order convergence rates are observed. When $\delta$ approaches zero slower than $h_{\textnormal{max}}$, ($\delta=\sqrt{h_{\textnormal{max}}}$), we observe first-order convergence rate. The numerical examples agree with \cref{thm:AC} and this verifies that the RK collocation method is asymptotically compatible. 



\begin{figure}[htb!]
\captionsetup[subfigure]{font=scriptsize, labelfont=scriptsize}
\begin{subfigure}[b]{0.5\textwidth} 
\centering
\scalebox{0.4}{\includegraphics{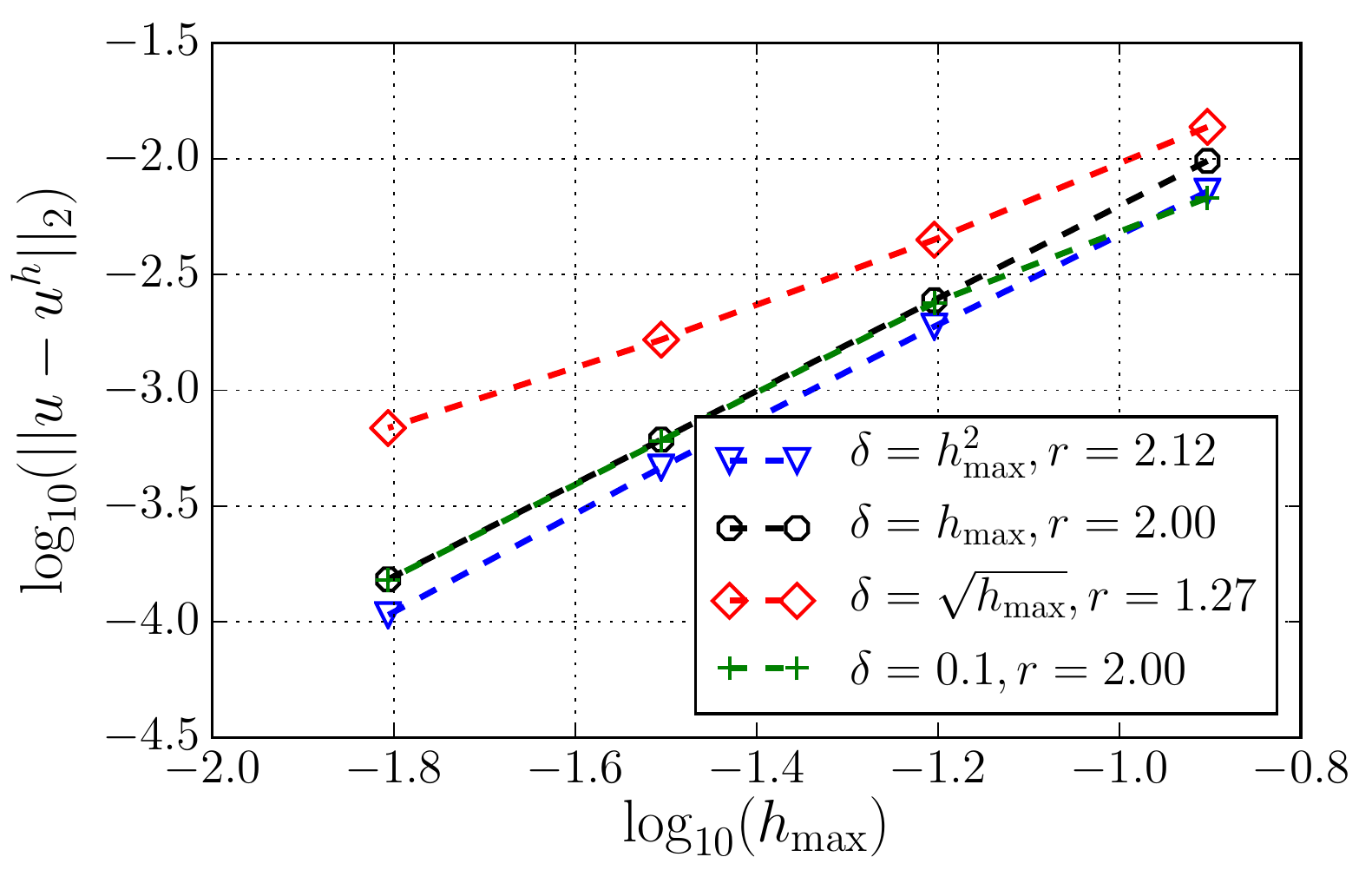}}
\caption{Gauss integration}
\label{fig:convergence}
\end{subfigure}
\hspace{-.5cm}
\begin{subfigure}[b]{0.5\textwidth} 
\centering
\scalebox{0.4}{\includegraphics{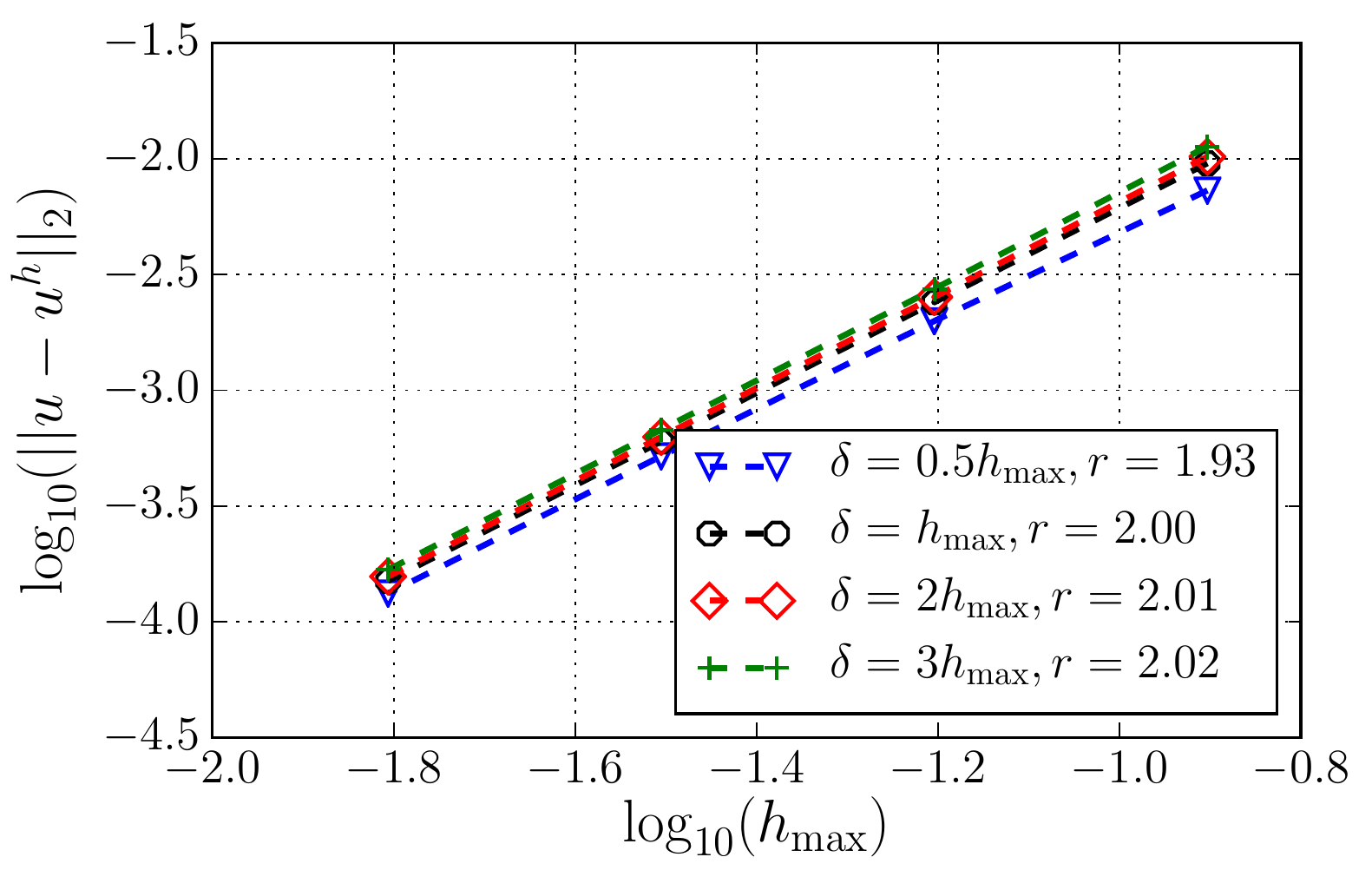}}
\caption{Meshfree integration}
\label{fig:MFconvergence}
\end{subfigure}
\caption{Convergence profiles using the RK collocation method.} 
\end{figure}

\subsection{RK collocation with meshfree integration}
\label{subsec:meshfree_numer}
As discussed in the previous section, high-order Gauss quadrature points are necessary to evaluate the integral for the RK collocation method, making the method computationally expensive. In practice, we sometimes couple grid size with horizon as $\delta=M_0 h_{\max}$ so that $\delta$ goes to $0$ at the same rate as $h$ approaches $0$. Now, we use the RK collocation method with meshfree integration as discussed in \cref{sec:QuasiDiscrete} to solve \cref{eqn:NTLEx} and only study the convergence to the local limit. We use $\delta = 3 \epsilon$ in this experiment where the $\delta$-neighborhood of each nodal point contains $29$ quadratic points (as depicted in \cref{fig:DiscreteBall}). Convergence profiles are presented in \cref{fig:MFconvergence}.  We observe a second-order convergence rate in agreement with \cref{thm:QAC}. Therefore the RK collocation method with meshfree integration converges to the correct local limit. 
This numerical experiment shows that the number of quadratic points can be significantly reduced for the case of small $\delta$ by using the meshfree integration techique.

\subsection{The test on the effect of boundary conditions}
\label{subsec:boundary_numer}
The numerical experiments in \cref{subsec:collocation_numer,subsec:meshfree_numer} are constructed with proper boundary conditions to ensure that exact solutions are sufficiently smooth across the boundary set as needed in the convergence theorems given in \cref{sec:CAnalysis,sec:DCAnalysis}. In this subsection, we use numerical experiments to show that improperly imposed boundary conditions might lead to reduced convergence rates.
We take $\Om=(0,1)^2$ and choose $u(x_1,x_2) = x_1^2x_2^2(1-x_1^2)(1-x_2^2)$ as the manufactured solution. The corresponding  $f_0= -\Delta u$ is given by
\[
f_0= (12x_1^2 -2)x_2^2(1-x_2^2) + (12x_2^2 -2)x_1^2(1-x_1^2),
\]
and $u(\bm{x}) = 0, \forall \bm{x} \in \partial\Om$. We study the convergence of the numerical solution of the nonlocal problem $-\cL_\del u = f_0$  to the local problem
\begin{equation} 
\label{eqn:NTLEx_local2}
\begin{cases}
-\Delta u(\bm{x}) = f_0(\bm{x}), & \bm{x} \in \Omega, \\
 \quad \quad u(\bm{x}) = 0, & \bm{x} \in  \partial \Omega,
\end{cases}
\end{equation}
as $\delta \to 0$. Nonlocal boundary conditions are imposed as follows.
\begin{remark} \label{rmk:imposebc}
There are two ways to impose the nonlocal boundary condition on $\Om_\cI$: 
\begin{enumerate}[label=(\roman*)] 
    \item $u(\bm{x}) = x_1^2x_2^2(1-x_1^2)(1-x_2^2), \quad \bm{x} \in \Om_\cI$; \label{eqn:bce} 
    \item $u(\bm{x}) = 0, \quad \bm{x} \in \Om_\cI $.  \label{eqn:bc0}
\end{enumerate}
\end{remark}

By imposing boundary condition as \cref{rmk:imposebc}\ref{eqn:bce} which is also conducted in \cref{subsec:collocation_numer,subsec:meshfree_numer}, we effectively assume the exact solution of \cref{eqn:NTLEx_local2} is extended smoothly outside the domain $\Om$ so that the regularity assumptions in \cref{thm:AC,thm:QAC} are satisfied. In contrast, using \cref{rmk:imposebc}\ref{eqn:bc0} as the boundary condition we violate the regularity assumptions even though the manufactured solution vanish on $\partial\Om$. Convergence profiles are presented in \cref{fig:gaussbc,fig:mfbc}. The results for the two collocation methods presented in \cref{sec:RKCollocation,sec:QuasiDiscrete} and are consistent.   If the boundary values are imposed exactly as the manufactured solution, i.e., \cref{rmk:imposebc}\ref{eqn:bce}, the convergence rates agree with our analysis, see \cref{fig:diffusione,fig:mfdiffusione}. On the other hand, the boundary condition given by \cref{rmk:imposebc}\ref{eqn:bc0} results in lower convergence rates as shown in \cref{fig:diffusion0,fig:mfdiffusion0}. Related studies on appropriate nonlocal boundary conditions in one dimension can be found in \cite{du2019uniform}. 

\begin{figure}[htb!]
\captionsetup[subfigure]{font=scriptsize, labelfont=scriptsize}
\begin{subfigure}[b]{0.5\textwidth} 
\centering
\scalebox{0.4}{\includegraphics{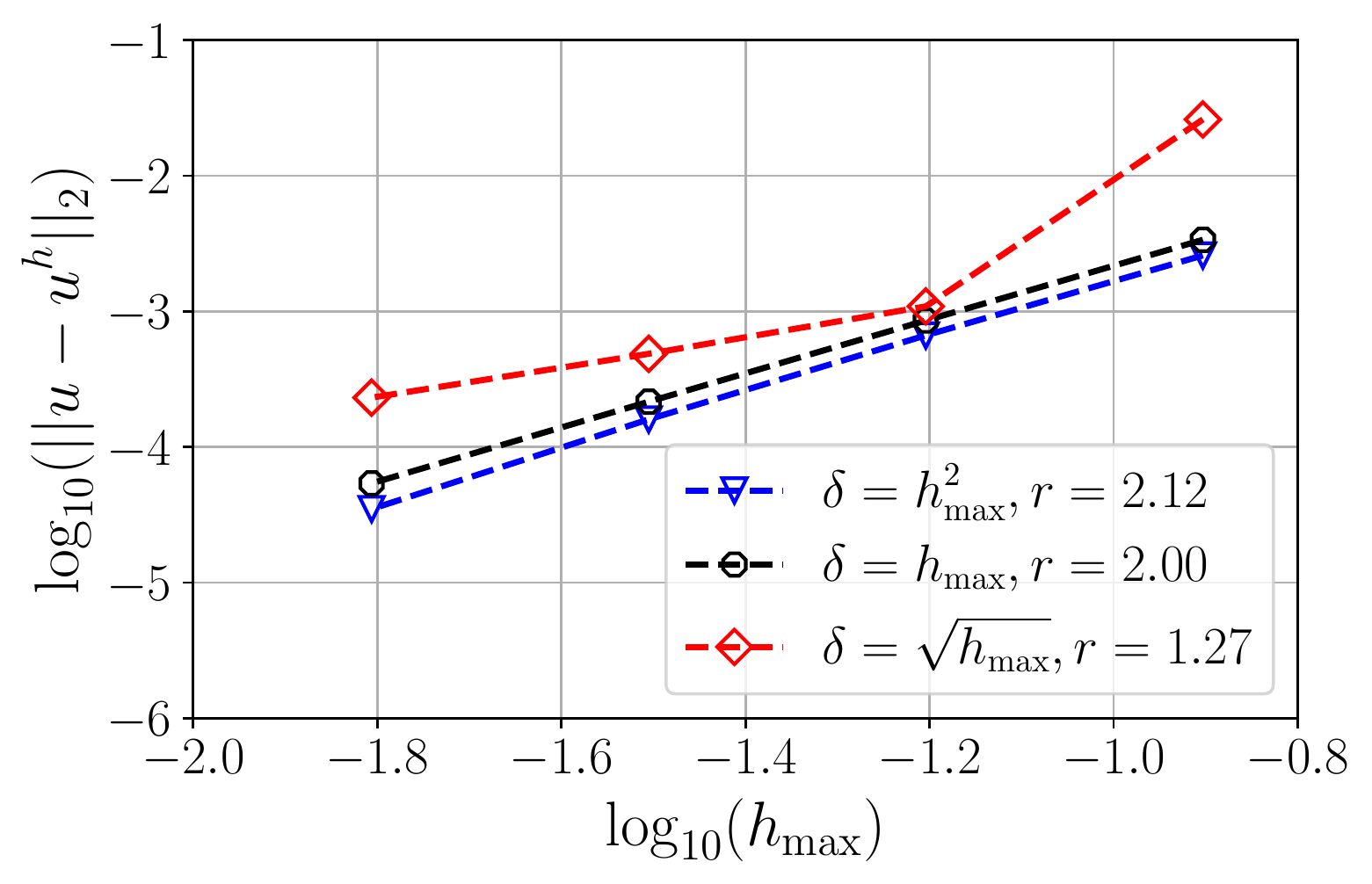}}
\caption{Impose boundary condition \\ as in \cref{rmk:imposebc}\ref{eqn:bce}}
\label{fig:diffusione}
\end{subfigure}
\hspace{-.5cm}
\begin{subfigure}[b]{0.5\textwidth} 
\centering
\scalebox{0.4}{\includegraphics{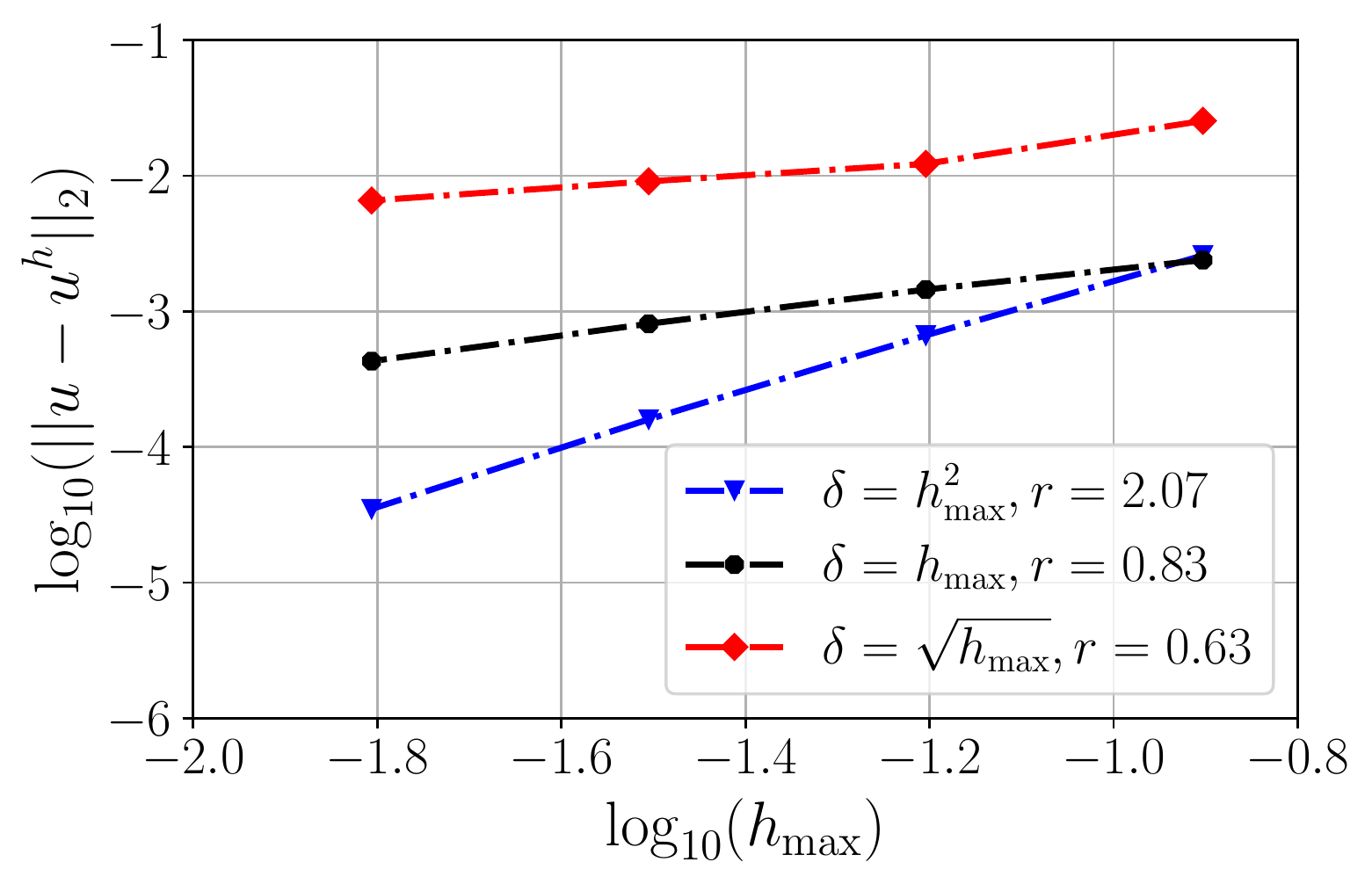}}
\caption{Impose boundary condition \\ as in \cref{rmk:imposebc}\ref{eqn:bc0}}
\label{fig:diffusion0}
\end{subfigure}
\caption{Convergence profiles using the RK collocation method with Gauss integration} 
\label{fig:gaussbc}
\end{figure}

\begin{figure}[htb!]
\captionsetup[subfigure]{font=scriptsize, labelfont=scriptsize}
\begin{subfigure}[b]{0.5\textwidth} 
\centering
\scalebox{0.4}{\includegraphics{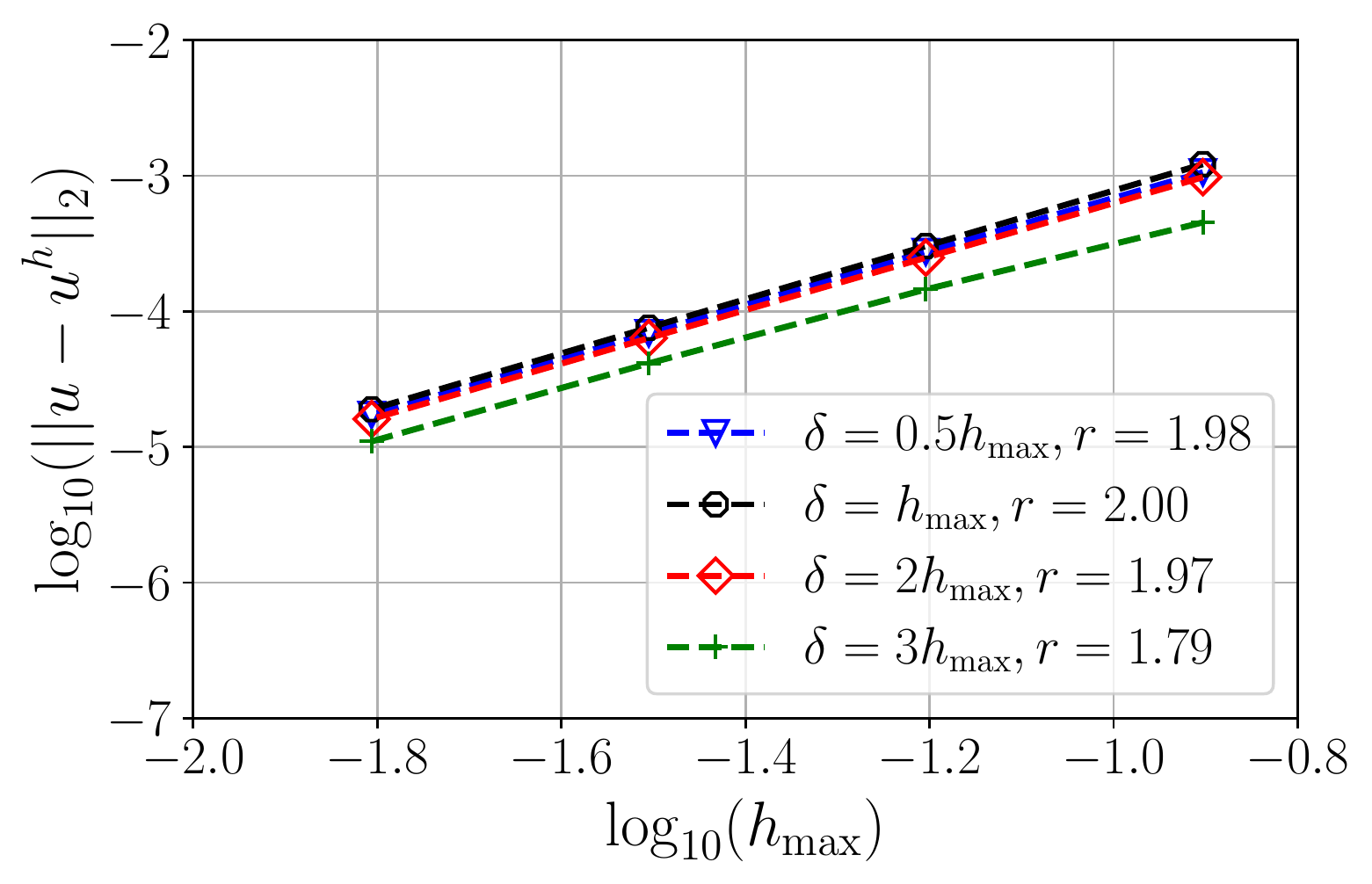}}
\caption{Impose boundary condition \\ as in \cref{rmk:imposebc}\ref{eqn:bce}}
\label{fig:mfdiffusione}
\end{subfigure}
\hspace{-.5cm}
\begin{subfigure}[b]{0.5\textwidth} 
\centering
\scalebox{0.4}{\includegraphics{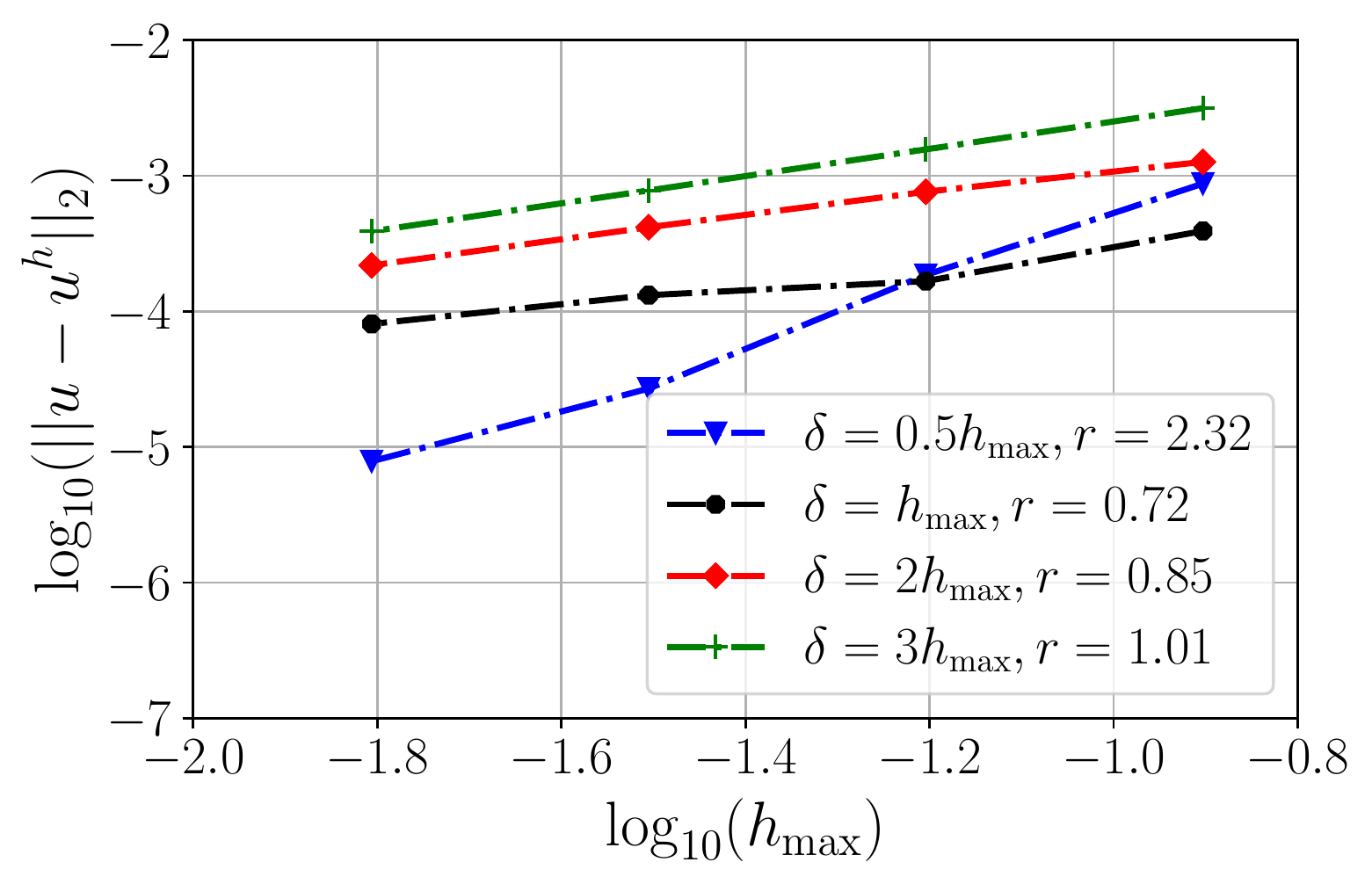}}
\caption{Impose boundary condition \\ as in \cref{rmk:imposebc}\ref{eqn:bc0}}
\label{fig:mfdiffusion0}
\end{subfigure}
\caption{Convergence profiles using the RK collocation method with meshfree integration} 
\label{fig:mfbc}
\end{figure}

\section{Conclusion} \label{sec:Conclusion}
In the first part of this work, we have presented an asymptotically compatible linear RK collocation method with special choices of RK support sizes for nonlocal diffusion models with Dirichlet boundary condition. Numerical solution of the method converges to both the nonlocal solution ($\delta$ fixed) and its local limit ($\delta \to 0$). We have provided stability analysis of this scheme in the case of Cartesian grids with varying resolution in each dimension. 
Since the standard Galerkin scheme  has been proved to be stable, the key idea to show the stability of the collocation scheme was to establish a relationship between the two schemes. Consistency of the collocation scheme is obtained by applying the properties of RK approximation.  
In the second part of this work, we have developed a quasi-discrete nonlocal diffusion operator using a meshfree integration technique,  where the main motivation is to reduce the computational cost by replacing the integral operator with a summation operator with only a few quadrature points inside the $\delta$-neighborhood of each point. The quadrature weights corresponding to the quadrature points are solved under polynomial reproducing conditions.
We unified two approaches, the RKPM and the GMLS approach, to calculate the quadrature weights. 
Under the assumption that the quadrature points to be symmetrically distributed inside the horizon,
we can show that the quadrature weights  are positive, which is crucial for the stability of the method. 
 The numerical solution of the RK collocation method applied to the quasi-discrete nonlocal diffusion operator was shown to converge to the correct local limit. 
Meanwhile, we validated our mathematical analysis by carrying out numerical examples in two dimensions. The order of convergence observed in the numerical examples match our theoretical results. That is, for the RK collocation method, the numerical solution converges to the nonlocal solution for a fixed $\delta$ and its local limit independent of the coupling of $\delta$ and discretization parameter $h_{\max}$; for the RK collocation method with meshfree integration and when the ratiao $\delta/h_{\max}$ is fixed, the numerical solution converges to the correct local limit. 


This work provides a rigorous analysis of collocation methods for nonlocal diffusion models, and there are several future directions needs to be mentioned. First, 
it is natural to extend the framework to the study of more general nonlocal models such as the peridynamics model of continuum mechanics and it is carried out in a separate work \cite{Leng2020}. 
In terms of analysis, the present work is restricted to the linear RK collocation method on the special grids, and we expect to study RK collocation methods more generally including high order methods with general meshes. 
Moreover, we remark that the error estimate is carried out with a strong assumption on the regularity of exact solutions. It is of great interest to improve the error estimate by a reduced regularity assumption on the exact solutions. 
There are also interesting computational work. 
For example, the mesh free integration technique, designed to reduce the computational cost, is only tested with shrinking horizon when the number of quadrature points in the $\delta$-neighborhood of each nodal point can be chosen as a small fixed number. It is also worthwhile to do a quantitative comparison in the future of the number of quadrature points needed in the case of a fixed nonlocal length $\delta$ using the Guass quadrature and meshfree integration technique. 

\section*{Acknowledgements}
The authors acknowledge the support of the SNL Laboratory Directed Research and Development (LDRD) program, and the SNL-UT academic alliance program.  The
Oden Institute is acknowledged for its support. The authors also thank Leszek Demkowicz, Qiang Du  and Marco Pasetto for helpful discussions on the subject. Last but not least, the authors thank the anonymous reviewers for their suggestions to improve the manuscript. 

\bibliographystyle{siam}
\bibliography{mybib2}

\end{document}